\documentclass{amsart}

\usepackage{cleveref}
\usepackage{cite}
\usepackage{url}
\usepackage{mathrsfs} 

\newtheorem{theorem}{Theorem}

\newtheorem{lemma}{Lemma}
\theoremstyle{definition}

\newtheorem{step}{\textsc{Step}}
\crefname{step}{Step}{Step}

\crefname{experiment}{Experiment}{Experiment}
\crefname{figure}{Figure}{Figure}

\title[Numerical reconstruction of radiative sourcest]{Numerical reconstruction of radiative sources from partial boundary measurements}

\author[H.~Fujiwara]{Hiroshi Fujiwara}
\address{Graduate School of Informatics,  Kyoto University, Yoshida Honmachi, Sakyo-ku, Kyoto 606-8501, Japan}
\email{fujiwara@acs.i.kyoto-u.ac.jp}

\author[K.~Sadiq]{Kamran Sadiq}

\address{Computational Science Center, Faculty of Mathematics, University of Vienna,  Oskar-Morgenstern-Platz 1, 1090 Vienna, Austria}
\email{kamran.sadiq@univie.ac.at}

\author[A.~Tamasan]{Alexandru Tamasan}
\address{Department of Mathematics, University of Central Florida, Orlando, 32816 Florida, USA}
\email{tamasan@math.ucf.edu}

\usepackage{braket,amsopn,amssymb}
\DeclareMathSymbol{\Real}{\mathalpha}{AMSb}{"52}
\DeclareMathSymbol{\C}{\mathalpha}{AMSb}{"43}
\DeclareMathSymbol{\Z}{\mathalpha}{AMSb}{'132}
\DeclareMathOperator{\co}{Conv}
\DeclareMathOperator{\Repart}{Re}
\DeclareMathOperator{\Impart}{Im}

\DeclareMathOperator{\Arctan}{Arctan}

\DeclareMathOperator{\Span}{span}
\DeclareMathOperator{\cond}{cond}
\DeclareMathOperator{\pv}{p\text{.}v\text{.}}

\providecommand{\norm}[1]{\left\lVert#1\right\rVert}
\newcommand{\mut}{\mu_\text{t}}
\newcommand{\mua}{\mu_\text{a}}
\newcommand{\mus}{\mu_\text{s}}

\def\sep{\:;\:}

\usepackage{tikz}



\begin{document}

\maketitle

\begin{abstract}
We consider an inverse source problem in the stationary radiative transport through an absorbing and scattering medium in two dimensions. Using the angularly resolved 
radiation measured on an arc of the boundary, we propose a numerical algorithm to recover the source in the convex hull of this arc. The method involves an unstable step of 
inverting a bounded operator whose range is not closed. 
We show that the continuity constant of the discretized inverse grows at most linearly with the discretization step, thus stabilizing the problem.
  Numerical examples presented show the effectiveness of the proposed method.
\end{abstract}

\section{Introduction}\label{sec:intro}

Let $\Omega$ be a two-dimensional convex domain with smooth boundary, $\Lambda$ be an arc on its boundary $\partial \Omega$, and $\Omega^{+}:=\co(\Lambda)\cap\Omega$ be its convex hull inside $\Omega$; see Figure~\ref{fig:domain}. In the stationary case, when generated by a source of radiation $q$ embedded in $\Omega$, the density  $I(z,\xi)$ of particles at $z \in \Omega$ moving in the direction $\xi \in S^1$ solves the radiative transport 
problem: for $(z,\xi)\in \Omega\times S^1$,
\begin{gather} 
\xi\cdot\nabla_z I(z,\xi) +\bigl(\mua(z)+\mus(z)\bigr) I(z,\xi) -\mus(z) \int_{S^1} p(z; \xi\cdot\xi')I(z,\xi') d\sigma_{\xi'} =q(z) ,  \label{eq:rte} \\
I(z,\xi) \bigm|_{\Gamma_{-}} = 0, \label{eq:noincoming}
\end{gather}
where  $\mua$ and  $\mus$ are respectively, the absorption and the scattering coefficients, and $p$ is the scattering phase function. 
The latter represents the probability at which particles change direction from $\xi'$ to $\xi$ due to scattering at $z$. In particular $\int_{S^1} p(z;\xi\cdot\xi')\:d\sigma_{\xi'} = 1$, with $d\sigma_{\xi'}$ denoting the arc element on $S^1$.
In \eqref{eq:noincoming} we distinguish the inflow boundary $\displaystyle\Gamma_{-} = \set{ (z,\xi) \in \partial\Omega\times S^1 \sep \nu(z)\cdot\xi < 0},$ 
where $\nu(z)$ is the outer unit normal at $z \in \partial\Omega$. The boundary condition \eqref{eq:noincoming} indicates that no radiation enters the domain from outside $\Omega$.
While in general the source $q$ may be directional dependent, in this work we consider an isotropic but inhomogeneous source.  
In what follows, the attenuation coefficient $\mut = \mua + \mus$ is also used.
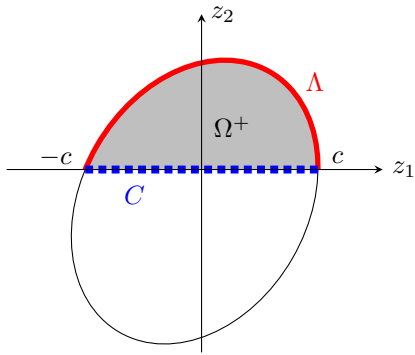
\begin{figure}[h]
\centering
\begin{tikzpicture}[>=stealth]

\draw (0,0) circle [x radius=1.5, y radius=2,rotate=-30];

\begin{scope}
\clip(-3,0.433) rectangle (3,2);
\draw [fill=lightgray,draw=red,line width=2pt] (0,0) circle [x radius=1.5, y radius=2,rotate=-30];
\end{scope}

\node [color=red] at (1.6,1.6) {$\Lambda$};


\draw [rotate=-30,line width=3pt,dotted,color=blue] (-1.4766,-0.35252) -- (1.2032,1.1947);
  \node [color=blue] at (-0.8,0.1) {$C$};
  \node at (-1.85,0.6) {$-c$};
  \node at (1.9,0.6) {$c$};

\node at (0.5,1) {$\Omega^{+}$};

\draw [->] (-2.5,0.433) -- (2.5,0.433) node [right] {$z_1$};
\draw [->] (0.092,-2) -- (0.092,2.5) node [right] {$z_2$};

\end{tikzpicture}
\caption{Exiting radiation is measured only on $\Lambda$.}\label{fig:domain}
\end{figure}

We are concerned with the following inverse problem: 
Reconstruct the unknown internal source $q$ in $\Omega^{+}$ from measurement of the outflow 
$I\lvert_{\Lambda_+}$ on $\Lambda_+$, where
\[
\Lambda_+ = \set{ (z,\xi) \in \Lambda\times S^1 \sep \nu(z)\cdot\xi > 0}.
\]
The medium (as characterized by $\mut$, $\mus$, and $p$)  is assumed known.

It is well known that the classical X-ray tomography turns into the inverse source problem to \eqref{eq:rte} if $\Omega$ is a non-scattering ($\mus=0$) and non-attenuating ($\mu_t=0$) medium \cite{bukhgeimBook,Romanov2017}.
More modern medical imaging techniques such as Positron/Single Photon Emission Tomography (PET/SPECT)  assume $\mu_a>0$ and $\mu_s=0$, whereas molecular imaging methods  \cite{PETKlose,Kim2005,Asllanaj2018} are modeled by the inverse source problem in scattering media ($\mus > 0$).


If the exiting radiation is known on the entire boundary (case $\Lambda=\partial\Omega$),  then $q$ is uniquely determined in $\Omega$ as shown in \cite{stefanovUhlmann} in the Euclidean domains  and \cite{sharafutdinov} on a simple Riemannian sufaces with small curvature. Reconstructions methods have also been proposed in \cite{balTamasan07}  in weakly scattering media and, for non-weakly scattering media in \cite{FujiwaraSadiqTamasan,FujiwaraSadiqTamasanSIIMS, eggerSchlottbom}, see also
\cite{smirnovKlibanovNguyen19} for slab domains and data on both sides of the boundary. 

For partial data, the singular support of the source (qualitative imaging)  can be recovered in a specific subdomain \cite{hubenthal}. Quantitative determination of the source 
$q\lvert_{\Omega^+}$ from data  on an arc of the boundary has been recently established by the authors in \cite{FujiwaraSadiqTamasanPartialIP}. The theoretical reconstruction method (summarized in Section \ref{sec:prelim}) is based on Bukhgeim's  theory of $A$-analytic functions \cite{bukhgeimBook}.

In this work we present and analyze an effective numerical algorithm, which reconstructs $q$ in ${\Omega^+}$ from exiting radiation on $\Lambda$. We know of no other method for quantitative imaging of a radiative source in stationary radiative transport, where data is collected on one side of the boundary.

Key to the method proposed in \cite{FujiwaraSadiqTamasanPartialIP} is the uniqueness of solution to the Cauchy type singular integral equation (CSIE)
\begin{equation}\label{CSIE}
[I-bH_c]f(x)=\Phi(x),\quad -c<x<c,\end{equation}
where $I$ is the identity operator, $b$ is a complex parameter, and 
\begin{equation}\label{eq:finiteHilbertTransform}
H_{c}[f](x) = \frac{1}{\pi}\pv\int_{-c}^{c} \frac{f(y)}{x-y}\:dy
\end{equation}
is the finite Hilbert transform of functions on $(-c,c)$.

Motivated by application to airfoil in aerodynamics and fracture mechanics in elasticity, numerous literature has studied CSIEs ~\cite{Muskhelishvili, Tricomi, Gakhov,ErdoganGupta,Ioakimidis1984,Goldberg1990}
for the case $b\neq i$. However, our problem  leads to consider \eqref{CSIE} for $b=i$. 
It is known  \cite{KoppelmanPincus1959,OkadaElliott1991} that the spectrum of $iH_c$ on $L^2(-c,c)$ is the interval $[-1,1]$. Fortunately, $1$ is not in the point spectrum of $iH_c$~\cite{widom1960,FujiwaraSadiqTamasanIPI}, thus allowing to invert $(I-iH_{c})$ in its range. 
Unfortunately, the range of $(I-iH_{c})$ is a dense proper subset of $L^2(-c,c)$ yielding an unstable inversion with the discretization schemes rendered ill-conditioned.
The spectrum analysis by itself does not shed light on how severe this ill-posedness can be. In Section  \ref{sec:CSIE}  we analyze this ill-posedness in the case of the piecewise constant approximation and estimate a degree of ill-conditionedness of our discretization scheme.


In Section \ref{sec:Alg} we formulate the numerical algorithm, which is implemented in the numerical experiments in Section \ref{sec:numstudy}.

\section{Preliminaries}\label{sec:prelim}

This section provides a brief review of the theoretical background used in our numerical reconstruction. For details we refer to \cite{FujiwaraSadiqTamasanIPI, FujiwaraSadiqTamasanPartialIP}.

We use the identification of spatial points $z = (z_1, z_2) \in \Real^2$ with their complex representation $z=z_1+iz_2 \in \C$,
and identify velocities $\xi \in S^1$ with $\theta \in \Real/2\pi$ by the standard polar coordinate $\xi = (\cos\theta,\sin\theta)$.
Derivatives
$\partial = (\partial_{z_1} - i\partial_{z_2})/2$ and
$\overline{\partial} = (\partial_{z_1} + i\partial_{z_2})/2$,
convert the advection operator to $\xi \cdot\nabla_z = e^{-i\theta}\overline{\partial} + e^{i\theta}\partial$.

The exact solution $I$ to the problem \eqref{eq:rte} and \eqref{eq:noincoming} has the Fourier expansion
with respect to $\xi$ in $L^2(S^1)$,
\[
I\bigl(z,\xi(\theta)\bigr) = \sum_{m \in \Z} I_{m}(z) e^{im\theta}.
\]
Our boundary data is equivalent to knowledge of the sequence 
$( I_{m}|_{\Lambda} \sep m \in \Z \bigr)$ on $\Lambda$.
Suppose that there exists a positive integer $M$ such that
$p$ is sufficiently well approximated by its  Fourier polynomial in the angular variable,
\begin{equation}\label{p=trig-poly}
p(z; \xi\cdot\xi') \approx \sum_{|m| \leq M} p_m(z) e^{im\theta},
\end{equation}
where 
$\theta =\arccos (\xi\cdot\xi')$. 

Since $q$ is isotropic, and $I$ is real-valued,
the radiative transport equation \eqref{eq:rte} decomposes into the infinite elliptic system in $\Omega$:
\begin{alignat}{3}
&\overline{\partial} I_{1} + \partial I_{-1} + \mut I_0 = 2\pi \mus p_0 I_0 + q, &\qquad& \label{eq:system:1}\\
&\overline{\partial} I_{-m} + \partial I_{-m-2} + \mut I_{-m-1} = 2\pi \mus p_{-m-1} I_{-m-1}, &\qquad& 0 \leq m \leq M-1, \label{eq:system:2}\\
&\overline{\partial} I_{-m} + \partial I_{-m-2} + \mut I_{-m-1} = 0, &\qquad& m \ge M,\label{eq:system:3}
\end{alignat}
Note that \eqref{eq:system:3} holds approximately, and it becomes exact for scattering kernels satisfying \eqref{p=trig-poly} with equality.

Let
\[
h[\mut](z,\xi) = D[\mut](z,\xi) - \dfrac{1}{2}(I-iH)R[\mut](z\cdot\xi^{\perp},\xi^{\perp}),
\]
where $\xi^\perp$ is the counter-clockwise rotation of $\xi$ by $\pi/2$,
$D$ is the divergent beam transform
\[
D[\mut](z,\xi) = \int_0^\infty \mut(z+t\xi)\:dt,
\]
$R$ is the Radon transform in $\Real^2$
\[
R[\mut](s,\xi) = \int_{\Real}\mut(s\xi+t\xi^{\perp})\:dt,
\]
and $H$ is the Hilbert transform
\[
H[f](s) = \dfrac{1}{\pi}\pv\int_{\Real}\dfrac{f(t)}{s-t}\:dt.
\]
The operator $h$ admits the Fourier expansions
\[
e^{-h[\mut](z,\xi(\theta))} = \displaystyle\sum_{m=0}^\infty \alpha_m(z) e^{im\theta},
\quad (z,\xi) \in \overline{\Omega}\times S^1,
\]
and
\[
e^{+h[\mut](z,\xi(\theta))} = \displaystyle\sum_{m=0}^\infty \beta_m(z) e^{im\theta},
\quad (z,\xi) \in \overline{\Omega}\times S^1.
\]
If we define
\begin{equation}\label{Jays}
J_{-m}(z) = \sum_{j=0}^\infty \alpha_j(z) I_{-m-j}(z),
\quad z \in \overline{\Omega}, m \ge 0,
\end{equation}
then its inversion is given by
$\displaystyle
I_{-m}(z) = \sum_{j=0}^\infty \beta_j(z) J_{-m-j}(z),
\quad z \in \overline{\Omega}, m \ge 0.
$

The sequence $\mathbf{J} = (J_{-m}|_{\Omega} \sep m \ge M)$ in \eqref{Jays} solves 
\begin{equation}\label{eq:L2analytic}
\overline{\partial} \mathbf{J} + \partial \mathcal{L}^2 \mathbf{J} = 0, \quad \text{in $\Omega$},
\end{equation}
where $\mathcal{L}$ denotes the left shift operator $\mathcal{L}(J_{-M}, J_{-M-1}, \dotsc) = (J_{-M-1}, J_{M-2}, \dotsc),$
as shown in \cite{sadiqTamasan01,SadiqScherzerTamasan}.

Solutions to \eqref{eq:L2analytic} are said to be $\mathcal{L}^2$-analytic in the sense of Bukhgeim~\cite{bukhgeimBook}. An important property of $\mathcal{L}^2$-analytic sequences is that they obey a Cauchy-like integral formula. 
For any simple  piecewise-smooth closed curve $\gamma\subset\overline{\Omega}$, the values of $\mathbf{J}$ at any point enclosed by $\gamma$ is determined by the values of  $\mathbf{J}$ on $\gamma$. More precisely, for $z$ enclosed 
by $\gamma$,
\begin{multline}\label{eq:CauchyBukhgeimInt}
J_{-m}(z)
= \dfrac{1}{2\pi i}\int_{\gamma} \dfrac{J_{-m}(\zeta)}{\zeta - z}d\zeta
\\
+ \dfrac{1}{2\pi i}\int_{\gamma} \left(\dfrac{d\zeta}{\zeta-z}-\dfrac{d\overline{\zeta}}{\overline{\zeta}-\overline{z}}\right)
\left\{\sum_{j=1}^\infty J_{-m-2j}(\zeta)\left(\dfrac{\overline{\zeta}-\overline{z}}{\zeta-z}\right)^j \right\}\:d\zeta,
\quad m \ge M.
\end{multline}

If we were to measure the outflow $I$ on the entire boundary $\partial\Omega$, then \eqref{Jays} would yield $(J_{-m}|_{\partial\Omega} \sep m \ge M)$ on the boundary, and \eqref{eq:CauchyBukhgeimInt} would give
$(J_{-m}|_{\Omega} \sep m \ge M)$ in $\Omega$, and then via \eqref{eq:system:1} the source $q$ would be obtained in $\Omega$.  This was the strategy used  in \cite{FujiwaraSadiqTamasan,FujiwaraSadiqTamasanSIIMS}. However, in our problem  $I$ is only known on the open arc $\Lambda$; see Figure \ref{fig:domain}.











Upon a rotation and translation of the domain, we may assume that the chord $C$ joining the endpoints of $\Lambda$ is the interval $(-c,c)$ on the real axis, for some $c>0$. We may also assume that $\Lambda$ lies in the upper half plane $\set{\text{Im}(z) > 0}$. 

Using \eqref{eq:CauchyBukhgeimInt} on $\Lambda \cup C$ and
taking the limit $\Omega^{+}\ni z\to x \in C$,
the second term in the right hand side
vanishes by virtue of continuity~\cite{sadiqTamasan01}.
Hence the Sokhotski-Plemelj formula yields the CSIE
\begin{equation}\label{eq:inteq:1}
(I-iH_{c})J_{-m}(x) = 2P^{-}[J_{-m}|_\Lambda](x), \quad x \in C,
\end{equation}
where
\begin{multline*}
P^{-}\bigl[J_{-m}|_{\Lambda}\bigr](z)
= \dfrac{i}{2\pi}\int_{\Lambda} \dfrac{J_{-m}(\zeta)}{z-\zeta}d\zeta
\\
- \dfrac{i}{2\pi}\int_{\Lambda} \left(\dfrac{d\zeta}{\zeta-z}-\dfrac{d\overline{\zeta}}{\ \overline{\zeta}-\overline{z}\ }\right)
  \sum_{j=1}^\infty J_{-m-2j}(\zeta)\left(\dfrac{\ \overline{\zeta}-\overline{z}\ }{\zeta-z}\right)^j,
\quad z \in \overline{\Omega^{+}}, m \ge M.
\end{multline*}

By solving \eqref{eq:inteq:1}, we determine the sequence $(J_{-m}|_{C} \sep m \ge M)$ on the chord $C$. Together with the data on $\Lambda$, the sequence $(J_{-m}| \sep m \ge M)$ is now known on $\partial\Omega_+$ and an application
of \eqref{eq:CauchyBukhgeimInt} yields its values in $\Omega_+$:
\begin{multline*}
J_{-m}(z)
= \dfrac{1}{2\pi i}\int_{-c}^{c} \dfrac{J_{-m}(t)}{t - z}dt
\\
+ \dfrac{1}{2\pi i}\int_{-c}^{c} \left(\dfrac{1}{t-z}-\dfrac{1}{t-\overline{z}}\right)
\left\{\sum_{j=1}^\infty J_{-m-2j}(t)\left(\dfrac{t-\overline{z}}{t-z}\right)^j \right\}\:dt
+ P^{-}\bigl[J_m|_{\Lambda}\bigr](z), \\
z \in \Omega^{+}, m \ge M.
\end{multline*}

The system \eqref{eq:system:2} yields a system of elliptic
boundary value problems for $I_{-m}$, $0 \leq m \leq M-1$, which can be solved iteratively in decreasing order of $m$ starting with $m=M-1$, and ending with $m=0$. More precisely, we obtain
\[
\begin{cases}
\overline{\partial} I_{-m} = f_{-m}, \quad \text{in $\Omega^+$, $0 \leq m \leq M-1$};\\
I_m|_{\Lambda} : \text{given},
\end{cases}
\]
where $f_{-m} = -\partial I_{-m-2} + (2\pi \mu_s p_{-m-1}-\mut)I_{-m-1}$. The Cauchy-Pompeiu formula~\cite{vekua} yields
\begin{equation}\label{eq:CauchyPompeiu}
I_{-m}(z)
= \dfrac{1}{2\pi i}\int_{\partial\Omega^{+}} \dfrac{I_{-m}(\zeta)}{\zeta-z}\:d\zeta
- \dfrac{1}{\pi}\int_{\Omega^{+}}\dfrac{f_{-m}(\xi+i\eta)}{(\xi+i\eta)-z}\:d\xi\:d\eta.
\end{equation}
Again, by taking the limit $\Omega^{+}\ni z \to x \in C$ and using the
Sokhotski-Plemelj formula, we obtain the CSIE
\begin{equation}\label{eq:inteq:2}
(I-iH_{c})I_{-m}(x)
= -\dfrac{2}{\pi}\int_{\Omega^+}\dfrac{f_{-m}(\xi+i\eta)\:d\xi\:d\eta}{(\xi-x)+i\eta}
+ \dfrac{1}{\pi i}\int_{\Lambda}\dfrac{I_{-m}(\zeta)}{\zeta-x}d\zeta, \\
\quad x \in C.
\end{equation}
This is again an equation like \eqref{eq:inteq:1} with a different right hand side.
The solution to the equation \eqref{eq:inteq:2} determines $I_{-m}$ on the chord $C$.  

Since the trace of $I_{-m}\lvert_{\Lambda}$ on $\Lambda$ is given for $0\leq m\leq M$,  we
have determined the trace  of $I_{-m}$ on the entire boundary of $\Omega_+$. The values of $I_{-m}$ in ${\Omega^{+}}$ can now be obtained 
by the Cauchy-Pompeiu formula \eqref{eq:CauchyPompeiu}:
\[
I_{-m}(z)
= -\dfrac{1}{\pi}\int_{\Omega^+}\dfrac{f_{-m}(\xi+i\eta)\:d\xi\:d\eta}{(\xi-z)+i\eta}
+ \dfrac{1}{2\pi i}\int_{\Lambda}\dfrac{I_{-m}(\zeta)}{\zeta-z}d\zeta
+ \dfrac{1}{2\pi i}\int_{-c}^{c} \dfrac{I_{-m}(x)}{x-z}\:dx.
\]

With $I_0$ and $I_{-1}$ now computed in $\Omega_+$, the source $q$ is determined via \eqref{eq:system:1}.

\section{On the numerical stability in a Singular Integral Equations of Cauchy type}\label{sec:CSIE}

A key step in our reconstruction procedure requires solving the Cauchy type singular integral equations (CSIE)  \eqref{eq:inteq:1} and \eqref{eq:inteq:2}, both of which are of the type 

\begin{equation}\label{eq:inteq}
\phi(x) - \dfrac{i}{\pi}\pv\int_{-c}^{c}\dfrac{\phi(y)}{x-y}dy = \Phi(x), \quad x \in (-c,c).
\end{equation}


Since the spectrum of $iH_c$ with the finite Hilbert transform $H_c$ in \eqref{eq:finiteHilbertTransform}
is $[-1,1]$, the operator $\bigl[(1+\epsilon)I-iH_{c}\bigr]^{-1}$ is bounded on $L^2(-c,c)$ for any $\epsilon>0$. This remark leads to the regularization of \eqref{eq:inteq} via solutions to 

\begin{equation}\label{eq:inteq:reg}
(1+\epsilon)\phi_\epsilon(x) - \dfrac{i}{\pi}\pv\int_{-c}^{c}\dfrac{\phi_\epsilon(y)}{x-y}dy = \Phi(x), \quad x \in (-c,c).
\end{equation}

One could approach solving \eqref{eq:inteq:reg} by considering solutions to 

\begin{equation}\label{eq:inteq:real}
(1+\epsilon)\phi_\epsilon(x) - \dfrac{b}{\pi}\pv\int_{-c}^{c}\dfrac{\phi_\epsilon(y)}{x-y}dy = \Phi(x), \quad x \in (-c,c),
\end{equation}and then letting $b\to i$. 

If $b$ is a real number, it is well known  \cite{Gakhov, Tricomi, Muskhelishvili}
that,  for any $\Phi \in L^2(-c,c)$ and $\epsilon > 0$, the equation \eqref{eq:inteq:real} has the unique solution given by
\begin{equation}\label{breal}
\phi_\epsilon(x) = \dfrac{(1+\epsilon) \Phi(x)}{(1+\epsilon)^2 + b^2}
- \dfrac{b}{\pi}\dfrac{e^{\tau(x)}}{(1+\epsilon)^2 + b^2}\int_{-c}^{c} \dfrac{e^{-\tau(y)}\Phi(y)}{y-x}dy,
\quad x \in (-c,c),
\end{equation}
where
\[
\tau(x) = -\dfrac{1}{\pi}\Arctan\dfrac{b}{1+\epsilon}\log\dfrac{c-x}{c+x}.
\]

The solution \eqref{breal} depends continuously on the parameter $b$ as long as  $(1+\epsilon)^2 + b^2 \neq 0$. 
For $\epsilon>0$ fixed, since $(1+\epsilon)^2 + b^2 \neq 0$ for $b$ along the segment  joining $0$ and $i$, the solution to
 \eqref{eq:inteq:reg} is obtained by setting  $b=i$ in \eqref{breal}. In the numerical evaluation we are then led to calculate
\begin{equation}\label{eq:numoscillation}
e^{\tau(x)}
= \cos\left(c_n \log\dfrac{c+x}{c-x}\right)
 + i\sin\left(c_n \log\dfrac{c+x}{c-x}\right),
\end{equation}
where
\[
c_n = \dfrac{1}{2\pi}\log\dfrac{2+\epsilon}{\epsilon} - n, \quad n \in \Z.
\]
However, infinite oscillations near the end points $x=\pm c$
make the numerical treatment of $e^{\tau}$ difficult. While one could try some particular numerical integration scheme, e.g., as in  \cite{DEformula},  
choosing optimal truncation parameters is non-trivial.

It is also known \cite{ErdoganGupta} that
the solution to \eqref{eq:inteq:real} can be approximated by
\[
\phi_\epsilon(x) \approx \sum_{n} C_n w(x/c) P_n^{(\mu,\nu)}(x/c),
\]
where
\[
\mu = -\nu = \dfrac{1}{2\pi i}\log\dfrac{2+\epsilon}{\epsilon}, \qquad
w(t) = (1-t)^\mu (1+t)^\nu,
\]
and $P^{(\mu,\nu)}_n(t)$ is the Jacobi polynomial.
In this algorithm the coefficients $C_n$ are given by
\[
C_n = -\dfrac{1}{b}\dfrac{\sin \pi\mu}{\theta^{(\mu,\nu)}_n}\int_{-1}^1
    P^{(-\mu,-\nu)}_n(t)w^{-1}(t)\Phi(ct)\:dt,
\]
with
\[
\theta^{(\mu,\nu)}_n
= \dfrac{2^{\mu+\nu+1}}{2n+\mu+\nu+1}
\dfrac{\Gamma(n+\mu+1)\Gamma(n+\nu+1)}{n!\:\Gamma(n+\mu+\nu+1)}.
\]
Unfortunately, due to the presence of the $1/w$ term in the integrand, the numerical integration in $C_n$ is as inefficient
as using the solution formula \eqref{eq:numoscillation}.

To circumvent these inefficiencies we propose to solve \eqref{eq:inteq:reg} via a  Galerkin approximation scheme.

Let $N$ be a positive integer,
$I_m = \bigl(-c+(m-1)\Delta x, -c+m\Delta x\bigr)$ denote equi-spaced intervals on $(-c,c)$ with $\Delta x = 2c/N$,
and $x_m$ be the mid-point of $I_m$.
Let $\chi_m(x)$ be the characteristic of $I_m$.
The approximation
\[
\phi_{\epsilon,\Delta} = \sum_{k=1}^N \phi_{\epsilon,k}\chi_k(x)
\]
to $\phi_\epsilon$  in $L^2(-c,c)$ by elements in $\Span\set{ \chi_m \sep 1 \leq m \leq N}$ reduces  \eqref{eq:inteq:reg}  to the semi-discrete equation
\begin{multline}\label{semidiscrete}
(1+\epsilon)\phi_{\epsilon,m}\Delta x
 - \dfrac{i}{\pi}\sum_{k=1}^N \phi_{\epsilon,k}
    \int_{-c}^c \left( \pv\int_{-c}^{c}\dfrac{\chi_k(y)\:dy}{x-y}\right)\chi_m(x)\:dx\\
 \approx \int_{-c}^{c}\Phi(x)\chi_m(x)\:dx, 
\quad 1 \leq m \leq N.
\end{multline}
Furthermore, by adopting the mid-point rule,
\begin{align*}
\int_{-c}^{c}\left( \pv\int_{-c}^{c}\dfrac{\chi_k(y)\:dy}{x-y}\right)\chi_m(x)\:dx
&\approx \pv\int_{-c}^{c}\dfrac{\chi_k(y)\:dy}{x_m-y}\Delta x
\\
&\approx \left(\int_{-c}^{x_m-\Delta x/2} + \int_{x_m+\Delta x/2}^{c}\right)\dfrac{\chi_k(y)\:dy}{x_m-y}\Delta x
\\
&= \begin{cases}
0, & \quad k = m;\\
\Delta x\displaystyle\int_{I_k}\dfrac{dy}{x_m-y} \approx \dfrac{\Delta x^2}{x_m - x_k}, &\quad k \neq m,
\end{cases}
\end{align*}
the equation \eqref{semidiscrete} reduces to the linear system
\begin{equation}\label{eq:inteq:d}
(1+\epsilon)\phi_{\epsilon,m} - \dfrac{i}{\pi} \sum_{k\neq m} \dfrac{\phi_{\epsilon,k}}{m-k} = \Phi_m,
\quad 1 \leq m \leq N,
\end{equation}
where
\[
\Phi_m = \dfrac{1}{\Delta x}\int_{I_m}\Phi(x)\:dx,
\quad\text{or}\quad \Phi(x_m).
\]
Note that \eqref{eq:inteq:d} could also be interpreted as the
discretization by the collocation method and the composite mid-point rule,
where $\phi_{\epsilon,m}$ would correspond to $\phi_\epsilon(x_m)$. 

Let $H_N$ be a square matrix of order $N$ whose $(m,k)$-entry is
\[
h_{mk} = \begin{cases}
\dfrac{1}{\pi}\dfrac{1}{m-k}, \quad & m\neq k;\\
0, \quad & m=k.
\end{cases}
\]
The norm of  $H_N$ in $\ell^2$ norm as a linear transformation on $\C^N$ is
known from Montgomery-Matthews' inequality~\cite{PhDMatthews, Matthews, Jameson} to be estimated by 
\begin{align}\label{MMinequality}
\norm{H_N}_2 \leq 1-\dfrac{1}{N}.
\end{align}
As a direct corollary, our proposed numerical procedure satisfies the following properties.

\begin{lemma}
For any $\epsilon \ge 0$ and
any positive integer $N$,
$(1+\epsilon)E_N-iH_N$ is strictly positive definite on $\C^N$,
where $E_N$ is the identity matrix of order $N$.
\end{lemma}
\begin{proof}
Noting that $(1+\epsilon)E_N-iH_N$ is Hermitian,
suppose that $\lambda \in \Real$ satisfies 
$\bigl[(1+\epsilon)E_N-iH_N\bigr]x = \lambda x$
with some $x \in \C^N\setminus\set{0}$.
Then, $iH_Nx = (1+\epsilon-\lambda)x$ and by the Montgomery-Matthews' inequality \eqref{MMinequality},
\[
|1+\epsilon-\lambda| \leq \norm{H_N}_2 \leq 1-\dfrac{1}{N}.
\]
Therefore
\[
0 < \epsilon + \dfrac{1}{N} \leq \lambda \leq 2+\epsilon-\dfrac{1}{N}.\qedhere
\]
\end{proof}

This directly guarantees that the algorithm proposed in the next section does not break down.

\begin{theorem}
For any $\epsilon \ge 0$, positive integer $N$,
and  $\begin{pmatrix} \Phi_m \end{pmatrix} \in \C^N$,
there exists a unique solution $\begin{pmatrix} \phi_{\epsilon,m} \end{pmatrix} \in \C^N$ to \eqref{eq:inteq:d}.
\end{theorem}

If $\epsilon = 0$, then
the condition number in $2$-norm is estimated as
\[
\cond_2\bigl(E_N - iH_N\bigr) \leq 2N-1,
\]
and it grows at most linearly with respect to $N$.
Furthermore for any $\epsilon > 0$,
\[
\cond_2\Bigl[(1+\epsilon)E_N - iH_N\Bigr]
\leq \dfrac{2+\epsilon - \tfrac{1}{N}}{\epsilon + \tfrac{1}{N}}
\leq \dfrac{2+\epsilon}{\epsilon}
\]
and, thus, is uniformly bounded with respect to $N$.

One may derive more accurate discretization to solve \eqref{eq:inteq:reg}
by choosing proper collocation points~\cite{ErdoganGupta} that 
adapt to the singularity of the solution at the end points,
or more accurate quadrature methods.
For example, 
if $\phi \in C^2$ in a neighborhood of  the chord $C$,
then the principal value integral can be evaluated as in \cite{HilbertTransform,FujiwaraSadiqTamasanSIIMS}, and then \eqref{eq:inteq:reg}
becomes
\[
(1+\epsilon)\phi_\epsilon(x) - \dfrac{i}{\pi}
\left(-\int_{-c}^c \psi(x,y)\:dy + \phi_\epsilon(x)\log\dfrac{c+x}{c-x}\right) = \Phi(x),
\]
with
\[
\psi(x,y) = \begin{cases}
\dfrac{\phi_\epsilon(x) - \phi_\epsilon(y)}{x-y}, \quad & x \neq y;\\
\phi_\epsilon'(x), \quad & x = y.
\end{cases}
\]Note that the integral on the left hand side is now in the sense of Riemann.
The composite mid-point rule and second-order approximations to $\phi_\epsilon'$
yields the linear system
\begin{multline}\label{eq:inteq:d2}
\left(1+\epsilon - \dfrac{i}{\pi}\log\dfrac{c+x_m}{c-x_m}
+ \dfrac{i}{\pi}\sum_{k\neq m}\dfrac{1}{m-k}\right) \phi_{\epsilon,m}
-  \dfrac{i}{\pi}\sum_{k\neq m}\dfrac{\phi_{\epsilon,k}}{m - k}
\\
+
\dfrac{i}{\pi}
\begin{Bmatrix}
-\tfrac{1}{2}\phi_{\epsilon,3} + 2\phi_{\epsilon,2} - \tfrac{3}{2}\phi_{\epsilon,1} \quad & m = 1\\[1ex]
\tfrac{1}{2}\phi_{\epsilon,m+1} - \tfrac{1}{2}\phi_{\epsilon,m-1} \quad & 1 < m < N\\[1ex]
\tfrac{1}{2}\phi_{\epsilon,N-2} - 2\phi_{\epsilon,N-1} + \tfrac{3}{2}\phi_{\epsilon,N} \quad & m = N
\end{Bmatrix}
 = \Phi_m, \quad 1 \leq m \leq N.
\end{multline}
\begin{figure}[h]
\centering
\includegraphics[width=.8\textwidth]{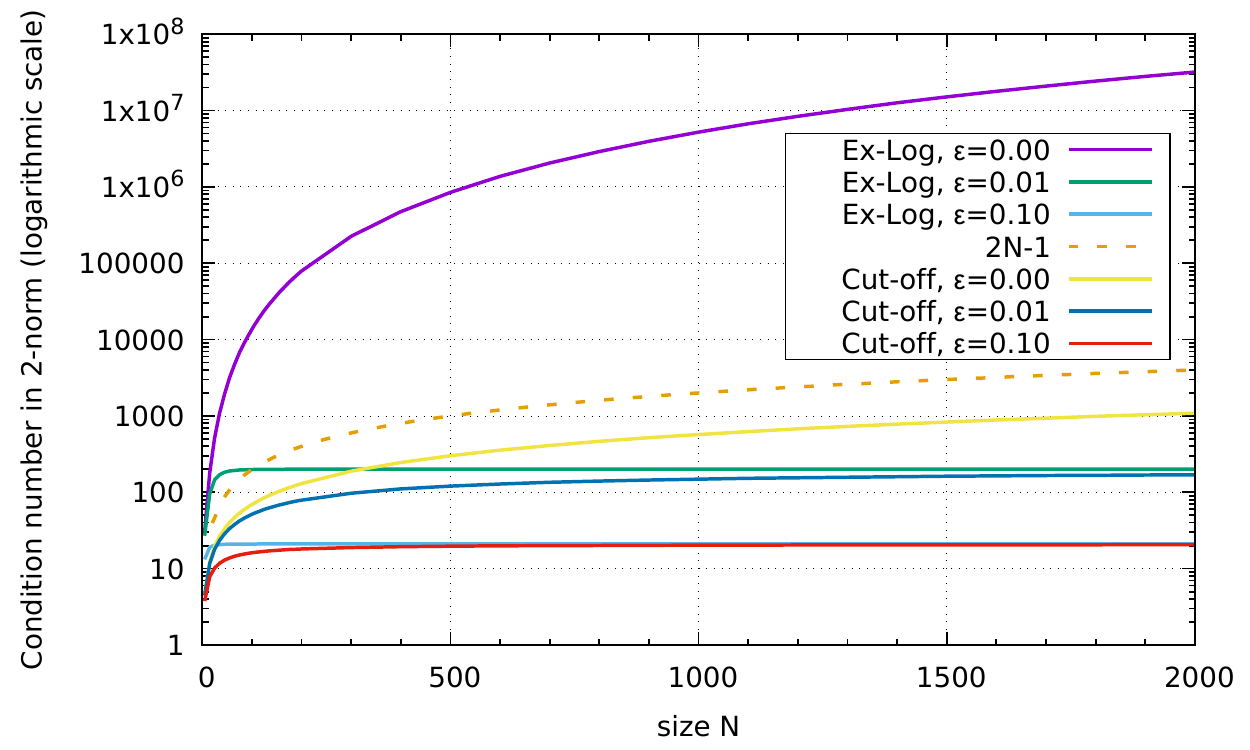}
\caption{\label{fig:cond}Condition numbers in $2$-norm
to linear systems by the cut-off method \eqref{eq:inteq:d} and the ex-log method \eqref{eq:inteq:d2}}
\end{figure}

To differentiate between the numerical method that uses  \eqref{eq:inteq:d} from the one that uses \eqref{eq:inteq:d2} in the numerical solvability of the singular integral equation \eqref{eq:inteq:reg}, we call the former
the \emph{cut-off method} and the  latter the \emph{extracting logarithmic singularity (ex-log) method}.


The condition numbers of \eqref{eq:inteq:d} and \eqref{eq:inteq:d2} for $\epsilon = 0$, $0.01$, and $0.1$
obtained numerically are depicted in Figure~\ref{fig:cond}
where the horizontal axis is the size of the matrix $N$ and
the vertical axis is the condition number in the logarithmic scale.
These results indicate that 
\eqref{eq:inteq:d} is numerically more stable than \eqref{eq:inteq:d2},
and the instability of \eqref{eq:inteq:d} with $\epsilon=0$ is not serious.

\section{Reconstruction Algorithm}\label{sec:Alg}
In this section, we present the numerical algorithm for the source reconstruction.

Recall that the absorbing and scattering medium (as characterized by $\mua$, $\mus$ in $\Omega$, and $p$ in $\Omega^{+}$) is known.
Assume that $\Lambda$ has a smooth parameterization $\zeta(\omega)$
for $\omega^{-} \leq \omega \leq \omega^{+}$, with $\zeta(\omega^{\mp}) = \pm c$.
Outflow through $\Lambda$, $I_{\text{measured}}(\zeta,\xi)$, is
sampled at $\bigl(\zeta_k, \xi(\theta_t)\bigr) \in \Gamma_{+}$
with $\zeta_k\in\Lambda$, where
$\zeta_1$, $\dotsc$, $\zeta_{K}$ are $K$ distinct points on $\Lambda$,
and $\theta_t = 2\pi t/T$ for a positive integer $T$.
We assume that $\zeta_k = \zeta(\tilde{\omega}_k)$ and
$\omega^{-} < \tilde{\omega}_1 < \tilde{\omega}_2 < \dotsb < \tilde{\omega}_{K} < \omega^{+}$.

\begin{step}
Let $\omega_k = (\tilde{\omega}_k + \tilde{\omega}_{k+1})/2$ for $1 \leq k \leq K-1$,
with $\omega_0 = \omega^{-}$ and $\omega_K = \omega^{+}$.
Write $\Delta\omega_k = \omega_{k+1} - \omega_{k}$ and
$\zeta_k' = \zeta'(\tilde{\omega}_k)$ for $1 \leq k \leq K$.
This leads an approximation by the composite mid-point rule
\[
\int_{\Lambda} f(\zeta)\:d\zeta
\approx \sum_{k=1}^K f(\zeta_k) \zeta_k' \Delta\omega_k.
\]
\end{step}

\begin{step}\label{step:omega}
We introduce an inscribed polygonal domain $\Omega^{+}_{\Delta} \approx \Omega^{+}$
whose closure includes $C$ and
take a triangulation $\mathcal{T} = \set{\tau_{\ell}}$ of $\Omega^{+}_{\Delta}$,
i.e.\ each $\tau_{\ell}$ is a triangular domain, $\tau_{\ell}\cap\tau_k = \emptyset$ if $\ell\neq k$,
and $\overline{\Omega^{+}_\Delta} = \displaystyle\bigcup_{\ell}\overline{\tau_{\ell}}$.
Let $P_1(\mathcal{T})$ denote the set of the piecewise linear continuous functions
with respect to $\mathcal{T}$.
We denote by $V$ the set of vertices of $\mathcal{T}$,
and by $r_\ell$ the centroid of $\tau_\ell$.
\end{step}

\begin{step}
Fix positive integers $N$, $M$ and $S$.
On the chord $C$, we allocate nodes $x_n = -c+(n-1/2)\Delta x$, $1 \leq n \leq N$
with $\Delta x = 2c/N$.
These are involved in the composite mid-point rule on $C$ as
\[
\int_C f(x)\:dx \approx \sum_{n=1}^N f(x_n)\Delta x.
\]
The integer $S$  should be chosen sufficiently large to truncate \eqref{eq:system:3},
at least $S \ge M+3$.
\end{step}

\begin{step}\label{step:IntegratingFactor}
For $0 \leq m \leq S$ and $\zeta_k \in \Lambda$, compute
\[
\alpha_{m,k} = \dfrac{1}{T}\sum_{t=0}^{T-1} \exp\Bigl(-h[\mut](\zeta_k,\xi(\theta_t))\Bigr) e^{-i(-m)\theta_t}.
\]
\end{step}

\begin{step}
For $0 \leq m \leq S-M-2$, compute
\begin{alignat*}{3}
\beta_{m,\ell} &= \dfrac{1}{T}\sum_{t=0}^{T-1} \exp\Bigl(h[\mut](z_\ell,\xi(\theta_t))\Bigr) e^{-i(-m)\theta_t},
&\quad& z_\ell \in V,
\intertext{and}
\beta^{C}_{m,n} &= \dfrac{1}{T}\sum_{t=0}^{T-1} \exp\Bigl(h[\mut](x_n,\xi(\theta_t))\Bigr) e^{-i(-m)\theta_t},
&\quad& x_n \in C. 
\end{alignat*}
The function $h[\mut]$ is evaluated by the use of mid-point rule with
a quadrature rule taking the singularity of the Cauchy kernel into account~\cite{FujiwaraSadiqTamasanSIIMS,HilbertTransform}.
\end{step}

\begin{step}\label{step:bd}
For $0 \leq m \leq S$, compute
\[
\mathcal{I}^{\Lambda}_{-m,k} = \dfrac{1}{T} \sum_{t=0}^{T-1} I_{\text{bd}}\bigl(\zeta_k,\xi(\theta_t)\bigr) e^{i(-m)\theta_t},
\quad \zeta_k \in \Lambda, 
\]
where
\[
I_{\text{bd}}\bigl(\zeta_k,\xi(\theta_t)\bigr) = \begin{cases}
I_{\text{measured}}\bigl(\zeta_k,\xi(\theta_t)\bigr), \quad & \text{on $\Gamma_{+}$};\\
0, \quad & \text{on $\Gamma_{-}$}.
\end{cases}
\]
Let $\mathcal{I}^{\Lambda}_{-m}$ a piecewise constant approximation:
$\mathcal{I}^{\Lambda}_{-m}\bigl(\zeta(\omega)\bigr) = \mathcal{I}^{\Lambda}_{-m,k}$, $\omega_{k-1} < \omega < \omega_k$.
\end{step}

\begin{step}\label{step:convolution}
For $M \leq m \leq S$, compute
\[
\mathcal{J}^{\Lambda}_{-m,k}
= \sum_{s=0}^{S-m} \alpha_{s,k} \mathcal{I}^{\Lambda}_{-m-s,k},
\quad \zeta_k \in \Lambda. 
\]
\end{step}

\begin{step}\label{step:solveCSIE1}
For $m=S-2,S-3,\dotsc,M$, find
the solution to the system of linear equations
\[
\bigl(E_N-iH_N\bigr)
\begin{pmatrix} \mathcal{J}^{C}_{-m,1} \\ \vdots \\ \mathcal{J}^{C}_{-m,N} \end{pmatrix}
= 
\begin{pmatrix} \mathcal{P}^{-}_{-m}(x_1) \\ \vdots \\ \mathcal{P}^{-}_{-m}(x_N) \end{pmatrix},
\]
where
\begin{multline*}
\mathcal{P}^{-}_{-m}(x)
 = \dfrac{i}{2\pi} \sum_{k=1}^{K}\dfrac{\zeta'_k\Delta\omega_k}{x-\zeta_k}\mathcal{J}^{\Lambda}_{-m,k}\\
 - \dfrac{i}{2\pi}\sum_{k=1}^{K}\left\{\Impart \left(\dfrac{\zeta'_k}{\zeta_k - x}\right)\Delta\omega_k\right\}
\left\{
  \sum_{m+2 \leq m+2j \leq S} \mathcal{J}^{\Lambda}_{-m-2j,k}\Biggl( \dfrac{\:\overline{\zeta_k} - \overline{x}\:}{\zeta_k - x} \Biggr)^j
\right\}.
\end{multline*}
\end{step}

\begin{step}\label{step:BoundaryIntegral}
Compute $\mathcal{J}^{\Omega^{+}}_{-m}(z_{\ell})$ for $M \leq m \leq S-2$ and $z_{\ell} \in V\cap\Omega^{+}$,
where
\begin{multline*}
\mathcal{J}^{\Omega^{+}}_{-m}(z)
= \mathcal{P}^{-}_{-m}(z)
 + \dfrac{\Delta x}{2\pi i} \sum_{n=1}^{N}\dfrac{1}{x_n - z}\mathcal{J}^{C}_{-m,n}\\
 + \dfrac{\Delta x}{2\pi i} \sum_{n=1}^{N}\left(\dfrac{1}{x_n-z} - \dfrac{1}{x_n-\overline{z}}\right)
\left\{
  \sum_{m+2 \leq m+2j \leq S} \mathcal{J}^{C}_{-m-2j,n}\Biggl( \dfrac{\:x_n - \overline{z}\:}{x_n - z} \Biggr)^j
\right\}.
\end{multline*}
\end{step}

\begin{step}\label{step:deconvolution}
For $z_{\ell} \in V\cap\Omega^{+}$, compute
\begin{align*}
\mathcal{I}^{\Omega^{+}}_{-M,\ell} &= \sum_{s=0}^{S-M-2} \beta_{s,\ell} \mathcal{J}^{\Omega^{+}}_{-M-s}(z_{\ell}),
\intertext{and}
\mathcal{I}^{\Omega^{+}}_{-M-1,\ell} &= \sum_{s=0}^{S-M-3} \beta_{s,\ell} \mathcal{J}^{\Omega^{+}}_{-M-1-s}(z_{\ell}).
\end{align*}
\end{step}

\begin{step}\label{step:deconvolution:boundary}
For $x_n \in C$, compute
\begin{align*}
\mathcal{I}^{C}_{-M,n} &= \sum_{s=0}^{S-M-2} \beta^{C}_{s,n} \mathcal{J}^{C}_{-M-s,n},
\intertext{and}
\mathcal{I}^{C}_{-M-1,n} &= \sum_{s=0}^{S-M-3} \beta^{C}_{s,n} \mathcal{J}^{C}_{-M-1-s,n}.
\end{align*}
For $m=M,M+1$, denote by $\mathcal{I}^{C}_{-m}$ the piecewise constant approximation on $C$:
$\mathcal{I}^{C}_{-m}(x) = \displaystyle\sum_{n=1}^N \mathcal{I}^{C}_{-m}\chi_n(x)$.
\end{step}

\begin{step}\label{step:boundary}
Find $\mathcal{I}^{\partial\Omega^{+}}_{-M} = \sum a_{j} \varphi_{j} \in P_1(\partial\Omega^{+}_\Delta)$
by the  $L^2$-best approximation to $\mathcal{I}^{\Lambda}_{-M}$ (of \cref{step:bd})
and $\mathcal{I}^{C}_{-M}$ (of \cref{step:deconvolution:boundary}),
where $\varphi_{j}$ is the periodic and piecewise linear continuous function on $\partial\Omega^{+}$
with $\varphi_{j}(v_k) = \delta_{jk}$ (Kronecker's delta) and $v_k\in V\cap\partial\Omega^{+}$.
For a more detailed example, see \cite{FujiwaraSadiqTamasanSIIMS}. Now, 
$\set{ \mathcal{I}_{-M,\ell}^{\Omega^{+}} \sep z_\ell \in V \cap\Omega^{+}}$ in \cref{step:deconvolution}
and 
$\mathcal{I}^{\partial\Omega^{+}}_{-M}$ uniquely determine $\mathcal{I}_{-M} \in P_1(\mathcal{T})$.

Similarly, we can obtain $\mathcal{I}_{-M-1} \in P_1(\mathcal{T})$.
\end{step}

\begin{step}\label{step:ellipticCauchy}
For $m=M-1, M-2, \dotsc, 1,0$ (in descending order),
find $\mathcal{I}_{m} \in P_1(\mathcal{T})$ as follows.
Assume that $\mathcal{I}_{-m-2}, \mathcal{I}_{-m-1} \in P_1(\mathcal{T})$ are already computed. 
If $\mathcal{I}_{-m-2}|_{\tau_{\ell}} = a_{\ell}x_1 + b_{\ell}x_2 + c_{\ell}$,
then $\partial\mathcal{I}_{-m-2}|_{\tau_{\ell}} = (a_{\ell} - ib_{\ell})/2$.\\For $\tau_\ell \in \mathcal{T}$ compute
\[
\mathcal{F}_{-m,\ell}
= -\partial\mathcal{I}_{-m-2}|_{\tau_{\ell}}
+ \{2\pi \mus(r_\ell) p_{-m-1}(r_\ell) - \mut(r_\ell)\}\mathcal{I}_{-m-1}(r_\ell).
\]
Compute for $x_n \in C$, 
\[
\mathcal{R}_{-m,n}
= -\dfrac{2}{\pi}\sum_{\tau_\ell \in \mathcal{T}}\dfrac{\mathcal{F}_{-m,\ell} |\tau_\ell|}{c_\ell - x_n}
+ \dfrac{1}{\pi i}\sum_{k=1}^{K}\dfrac{\zeta'_k\Delta\omega_k}{\zeta_k - x_n}\mathcal{I}^{\Lambda}_{-m,k},
\]
where $|\tau_\ell|$ is the area of $\tau_\ell$.
Then solve the linear equation
\[
\bigl(E_N-iH_N\bigr)
\begin{pmatrix} \mathcal{I}^{C}_{-m,1} \\ \vdots \\ \mathcal{I}^{C}_{-m,N} \end{pmatrix}
= 
\begin{pmatrix} \mathcal{R}_{-m,1} \\ \vdots \\ \mathcal{R}_{-m,N} \end{pmatrix}.
\]
And compute for $z_\ell \in V\cap\Omega^{+}$,
\[
\mathcal{I}^{\Omega^{+}}_{-m,\ell}
= -\dfrac{1}{\pi}\sum_{\tau_{j} \in \mathcal{T}} \dfrac{\mathcal{F}_{-m,j}|\tau_{j}|}{c_j-z_{\ell}}
+ \dfrac{1}{2\pi i}\sum_{k=1}^{K}\dfrac{\zeta'_k\Delta\omega_k}{\zeta_k - z_\ell}\mathcal{I}^{\Lambda}_{-m,k}
+ \dfrac{\Delta x}{2\pi i}\sum_{n=1}^{N}\dfrac{\mathcal{I}^{C}_{-m,n}}{x_n - z_\ell}.
\]
For $z_\ell \in V\cap\partial\Omega^{+}$,
find $\mathcal{I}^{\partial\Omega^{+}}_{-m} \in P_1(\partial\Omega^{+}_{\Delta})$
by the best approximation in the $L^2$ sense to $\mathcal{I}^{\Lambda}_{-m}$ and $\mathcal{I}^{C}_{-m}$
similarly to \cref{step:boundary}.

These $\set{ \mathcal{I}^{\Omega^{+}}_{-m,\ell} \sep z_\ell \in V\cap\Omega^{+} }$ and $\mathcal{I}^{\partial\Omega^{+}}_{-m}$
give $\mathcal{I}_{-m} \in P_1(\mathcal{T})$.
\end{step}

\begin{step}
For each triangle $\tau_{\ell} \in \mathcal{T}$,
let $\mathcal{I}_{-1}|_{\tau_{\ell}} = a_{\ell} x_1 + b_{\ell} x_2 + c_{\ell}$.
The reconstruction of $q|_{\tau_{\ell}}$ is given by \eqref{eq:system:1},
\[
q_{\ell}
=
\Repart(a_{\ell}) + \Impart(b_{\ell})
 + \bigl\{\mut(r_{\ell})-2\pi\mus(r_{\ell})p_0(r_{\ell})\bigr\}\Repart\bigl(\mathcal{I}_0(r_{\ell})\bigr).
\]
This ends the Algorithm.
\end{step}

\bigskip

We remark here on the dual role played by the truncation parameter $M$ as a regularization parameter to control both accuracy and stability. 
On the one hand $M$  sets the degree of the trig-polynomial (in the angular variable) approximation of the scattering phase function $p(z,\; \cdot)$ influencing the accuracy. On the other hand, the terms after the $M$-th mode in the system \eqref{eq:system:2} are truncated, thus resulting \eqref{eq:system:3}. This also leads to truncation of higher frequency modes of the solution, which, in general help the stabilization of the numerical procedure.  
Since the present algorithm gives a point-wise reconstruction, one may choose a {locally optimal} $M$, e.g., by observing (for several values of $M$) the degradation of accuracy in integration due to the singularities of the Cauchy kernel near the boundary.

\section{Numerical Experiments}\label{sec:numstudy}

In this section the proposed algorithm is demonstrated to show its validity.
In particular, numerical results by the cut-off method \eqref{eq:inteq:d} and the ex-log method \eqref{eq:inteq:d2} are compared to solve CSIEs.
An example of choice of parameters are also exhibited.
Throughout the section, all computations are processed on EPYC 7643 with the IEEE754 double precision arithmetic.

Revisiting the example in \cite{FujiwaraSadiqTamasanPartialIP},
suppose that $\Omega$ is the unit disk with inclusions
\begin{align*}
  B_1 &= \set{ (z_1,z_2) \sep (z_1-0.5)^2 + z_2^2 < 0.3^2 }, \\
  B_2 &= \left\{ (z_1,z_2) \sep (z_1+0.25)^2 + \left(z_2-\dfrac{\sqrt{3}}{4}\right)^2 < 0.2^2\right\}, \\
  B_3 &= \set{ (z_1,z_2) \sep z_1^2 + (z_2+0.6)^2 < 0.3^2 },\\
\intertext{and}
  R &= \{ (z_1,z_2) \sep -0.25 < z_1 < 0.5, |z_2| < 0.15 
 \};  \text{ see Figure~\ref{fig:configuration}}.
\end{align*}
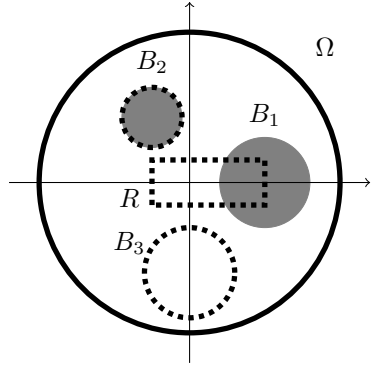
\begin{figure}[h]
  \begin{tikzpicture}[scale=2]
\draw [line width=2] (0,0) circle (1); 
  \node at (0.9,0.9) {$\Omega$};

\draw [color=gray,fill=gray,line width=0] (0.5,0) circle (0.3); 
  \node at (0.5,0.45) {$B_1$};

\draw [color=gray,fill=gray,line width=0] (-0.25,0.433013) circle (0.2); 
    \node at (-0.25,0.8) {$B_2$};

\draw [dotted,line width=2] (-0.25,0.433013) circle (0.2); 
\draw [dotted,line width=2] (0,-0.6) circle (0.3); 
  \node at (-0.4,-0.4) {$B_3$};

\draw [dotted,line width=2] (-0.25,-0.15) rectangle (0.5,0.15); 
  \node at (-0.4,-0.1) {$R$};

\draw [->] (-1.2,0) -- (1.2,0);
\draw [->] (0,-1.2) -- (0,1.2);

\end{tikzpicture}
  \caption{\label{fig:configuration}An absorbing and scattering domain $\Omega$: highly absorbing areas are $B_1$ and $B_2$;  the internal source is supported in $B_2\cup B_3\cup R$.}
\end{figure}
The internal source $q$ (to be recovered in $\Omega^{+}$ in our inverse problem) is given by
\[
q(z) =
  \begin{cases}
    2, \qquad & \text{in $R$};\\
    1, \qquad & \text{in $B_2 \cup B_3$};\\
    0, \qquad & \text{otherwise}.
  \end{cases}
\]

The attenuation coefficient $\displaystyle\mut=\mua+\mus$ accounts for the absorption (via $\mua$) and for the scattering off of the direction of counting (via $\mus$). In our numerical experiment $\mus = 3$ in $\Omega$ and
\[
\mua(z) =
  \begin{cases}
    2, \qquad & \text{in $B_1$};\\
    1, \qquad & \text{in $B_2$};\\
    0.1, \qquad & \text{otherwise}.
  \end{cases}
\]
The scattering kernel is the two dimensional Henyey-Greenstein (Poisson) kernel
\[
p(z; \xi,\xi') = \dfrac{1}{2\pi}\dfrac{1-g^2}{1-(\xi\cdot\xi')g + g^2},
\]
with $g=0.5$.

Measurement data on the boundary arc $\Lambda$ is obtained by the numerical computation of the forward problem \eqref{eq:rte} and \eqref{eq:noincoming} with the above choice of coefficients. We use the discontinuous Galerkin method with piecewise constant basis~\cite{fujiwaraDGRTE},
where $\Omega$ and $S^1\approx \Real/2\pi$ are respectively divided into $27{,}407{,}104$ triangular domains and $360$ intervals with equal length.
The outflow is measured at 3,141 equi-spaced nodes $\zeta\in\Lambda$ in direction $\xi\in S^1$ at $1$ radian intervals with $(\zeta,\xi) \in \Gamma_{+}$.
Figure~\ref{fig:measurement} illustrates the computed outflow through a few points on the boundary arc $\Lambda$,
where dependency on $\xi$ is represented in the polar coordinate with the center $\zeta$ indicated by cross symbols ($\times$) and radius $2I(\zeta,\xi)$.
\begin{figure}[th]
\begin{minipage}{.5\textwidth}
\centering
  \includegraphics[width=.8\textwidth,bb=220 140 820 680]{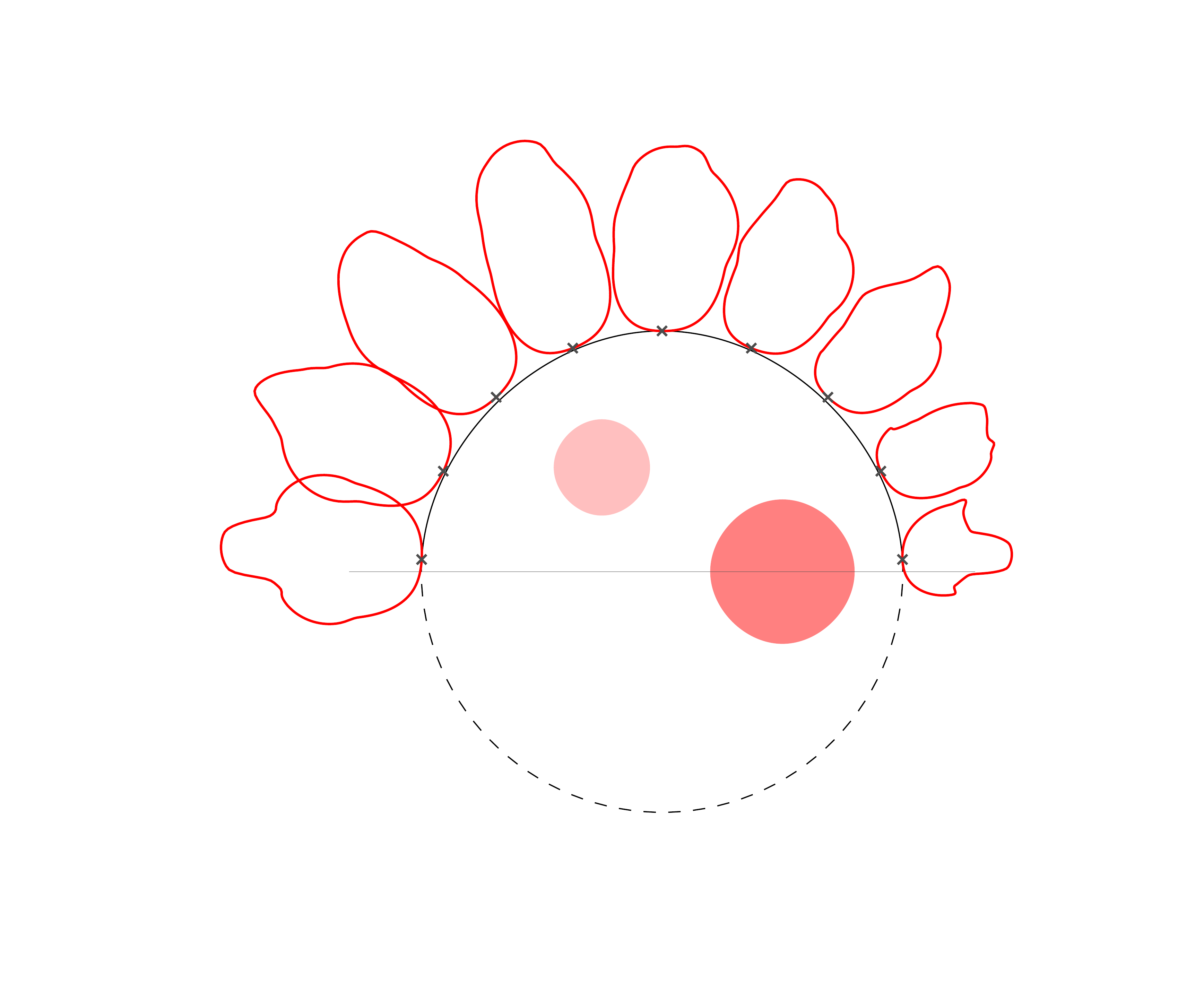}
\end{minipage}
\begin{minipage}{.3\textwidth}
  \includegraphics[width=.9\textwidth,bb=75 0 275 175]{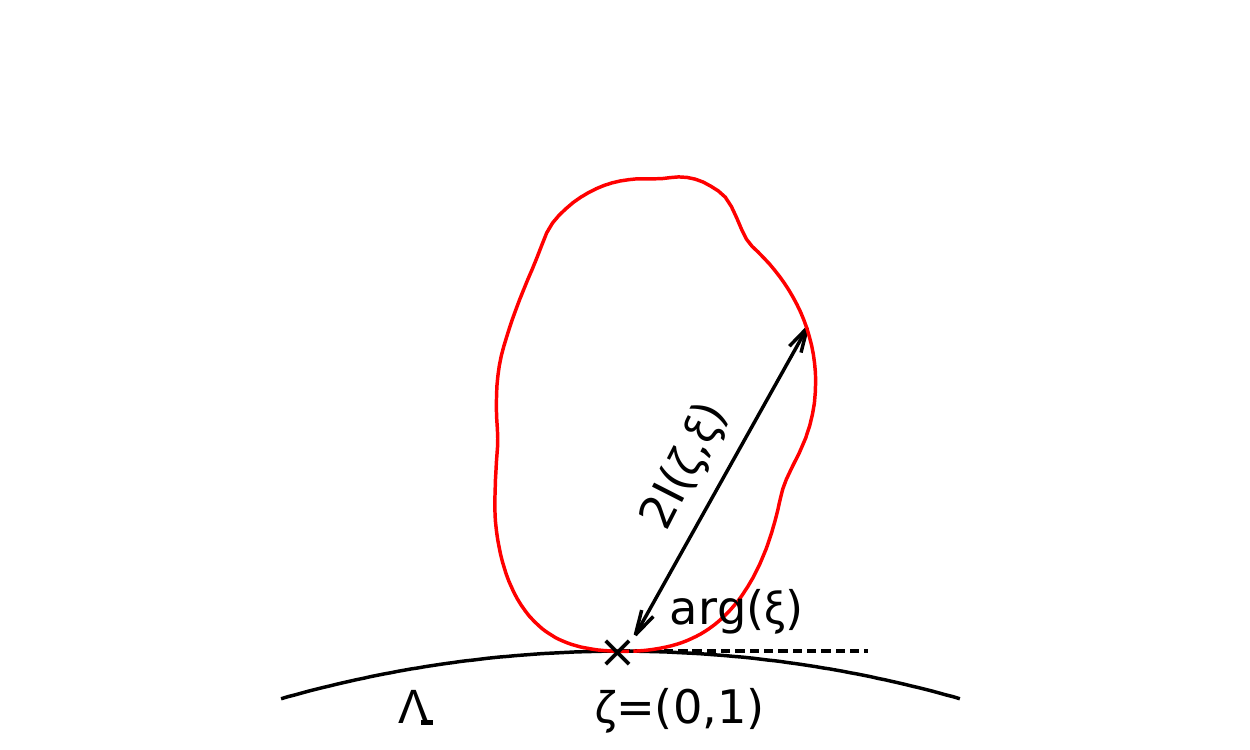}
\end{minipage}
\caption{\label{fig:measurement}Measurement data $I|_{\Lambda_+}$. The red curves depict the outflow $I(\zeta,\xi)$ in polar coordinate centered at $\zeta$ (cross symbol $\times$) and radius $2I(\zeta,\xi)$. The outflow through $\zeta=(0,1)$ is magnified in the right figure. The red inclusions show the highly absorbing regions given as a priori information (the medium is assumed known).}
\end{figure}

In the reconstruction \cref{step:omega}, the domain of interest $\Omega^{+}$ is approximated by $\Omega^{+}_{\Delta}$ and  consists of $8{,}631$ triangles. The number of triangles used in the reconstruction is much smaller than the one used in solving the forward problem, and they are not a sup-partition. In particular, the numerical experiments avoid an inverse crime.

In order to determine the truncation parameter $S$ in Step 3, the decay of $\alpha_m$ and $\beta_m$ are examined. Figure~\ref{fig:alpha_beta} presents $\max\{ |\alpha_{m,k}| \sep \zeta_k \in \Lambda \}$ and $\max\{ |\beta_{m,\ell}| \sep r_\ell \in \Omega^{+} \}$, where horizontal axis is the mode $m$. The Hilbert transform is computed by the method in ~\cite{FujiwaraSadiqTamasanSIIMS,HilbertTransform} with the composite mid-point rule. A comparison of two discretizations is shown in Figure~\ref{fig:alpha_beta}: The coarse discretization  (indicated by the green  $+$ symbols) uses $100$ spatial nodes for the Hilbert transform and $360$ angles for the Fourier series, whereas the fine discretization (shown in purple  $\times$ symbols) uses $10{,}000$ spatial nodes and $2{,}880$ angles. 
Both discretizations yield numerical results in good agreement with each other, and $\norm{\alpha_m} < 0.01$ and $\norm{\beta_m} < 0.01$ for $m \ge 175$.  By taking the number of measured directions into account, we adopt $S=175$ and use the coarse discretization.
\begin{figure}[ht]
  \includegraphics[width=.48\textwidth]{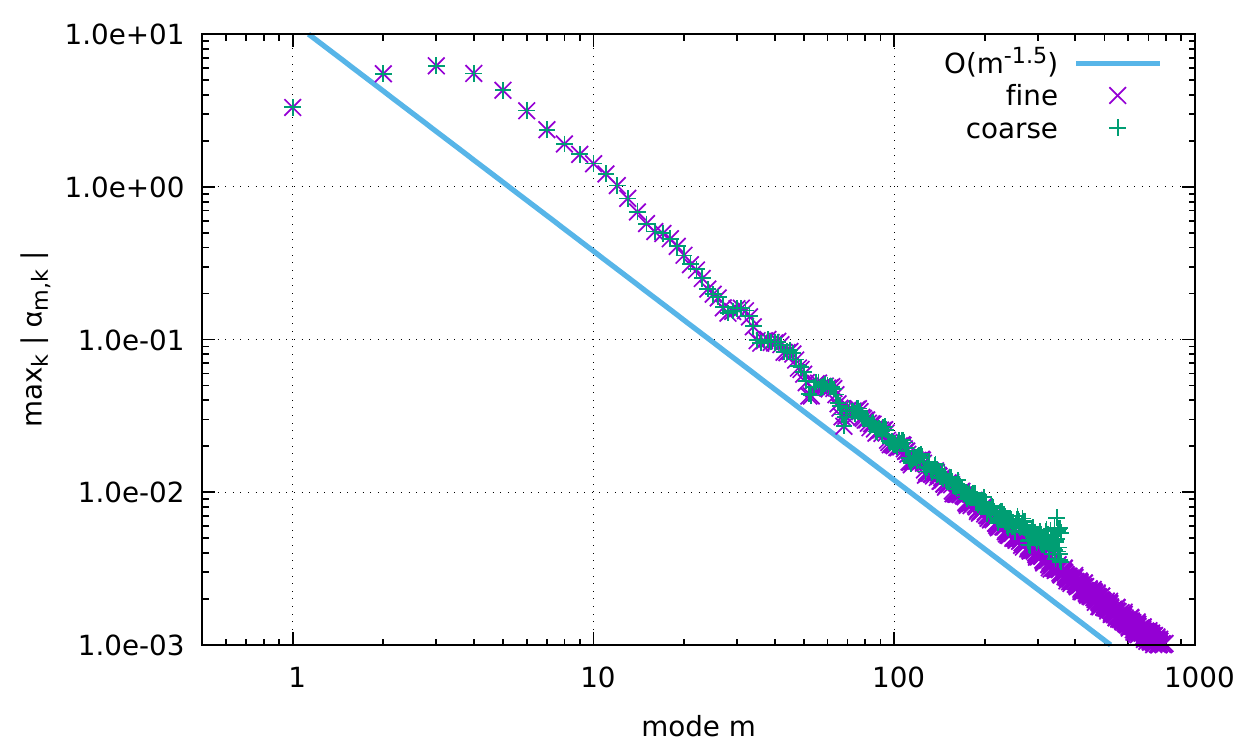}
  \includegraphics[width=.48\textwidth]{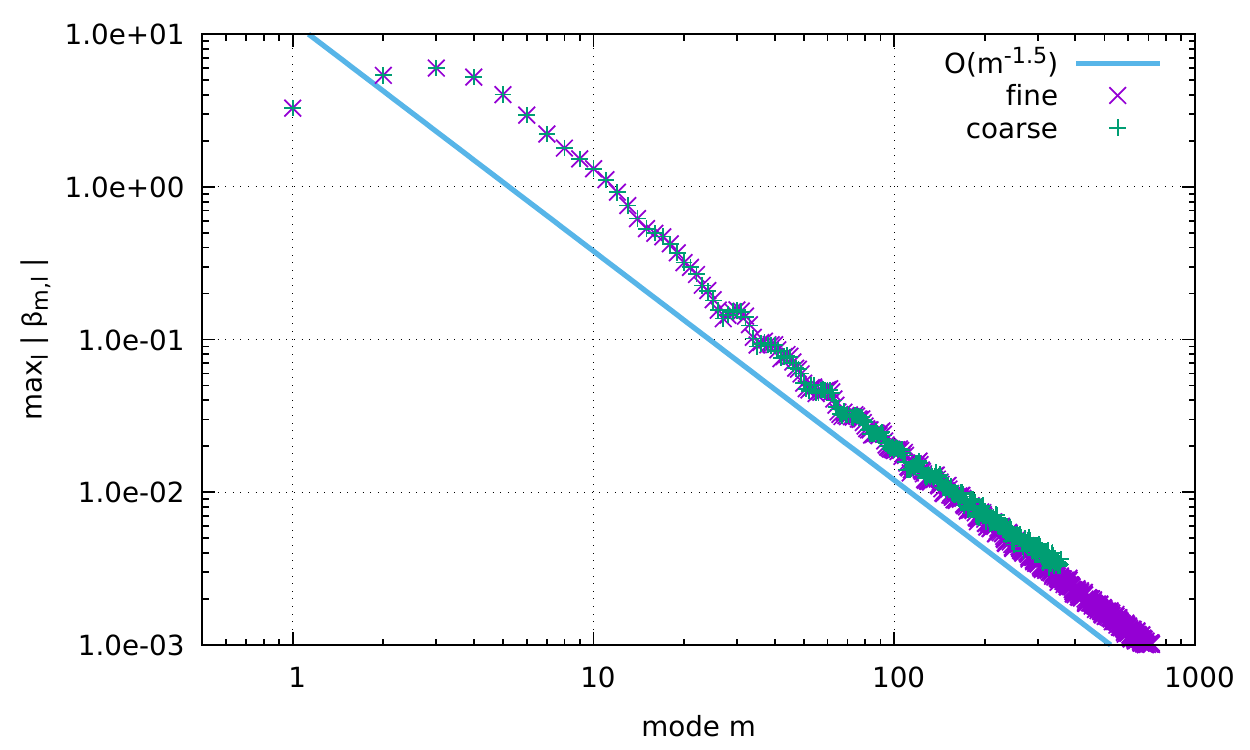}
\caption{\label{fig:alpha_beta}Decay of $\norm{\alpha_{m}}$ (left) and $\norm{\beta_{m}}$ (right). Green plus ($+$) symbols are results by coarse discretization, while purple cross ($\times$) are those by fine discretization.}
\end{figure}

To choose an optimal truncation parameter $M$ we use the criterion proposed in \cite{FujiwaraSadiqTamasanSIIMS}. Namely, we compute the $L^2$-norms of the imaginary part of the reconstructed $q$ for varying values of $M$, and choose the one corresponding to the smallest norm,  see Figure~\ref{fig:M}. 
When the cut-off method  with $\epsilon=0$ is used, then $\norm{\Impart q}_2$ attains the minimum for $M=10$,
while $M=12$ gives the minimum for the ex-log method with $\epsilon=0.01$.

Note that the parameters $M$ and $S$ are optimized without the use of the solution to the forward problem.

\begin{figure}[ht]
  \includegraphics[width=.48\textwidth]{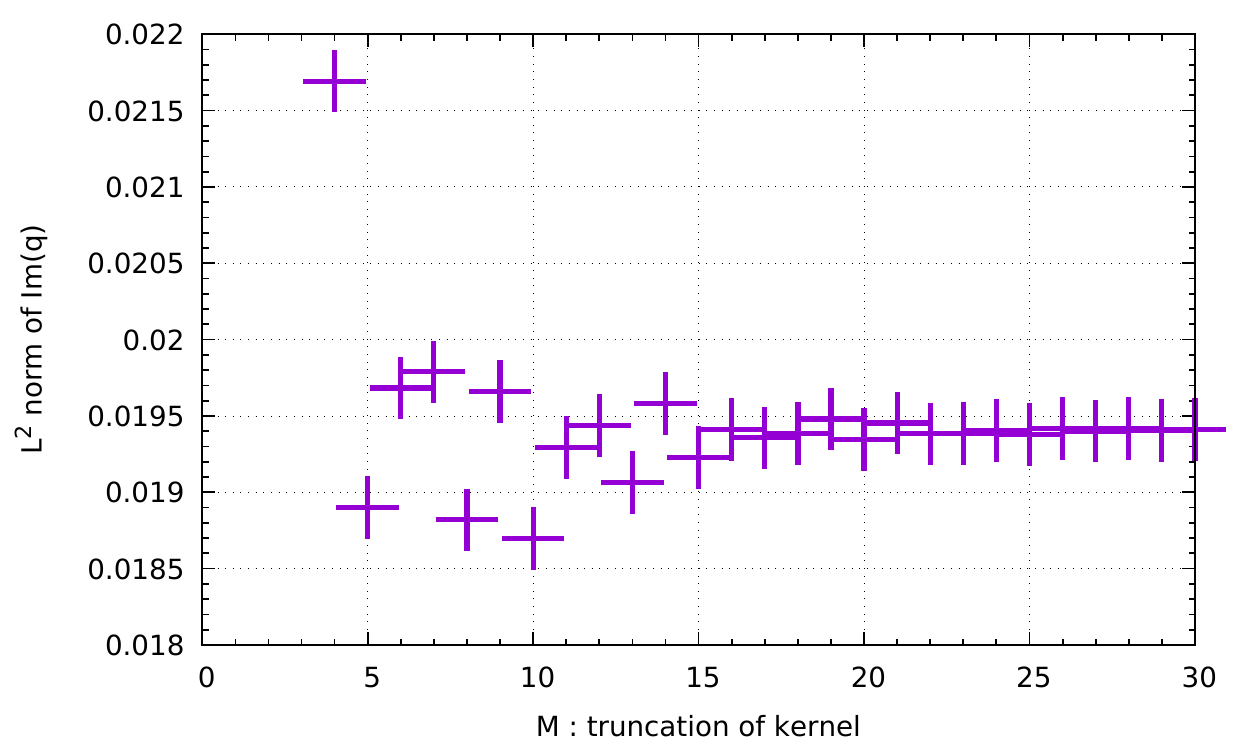}
  \includegraphics[width=.48\textwidth]{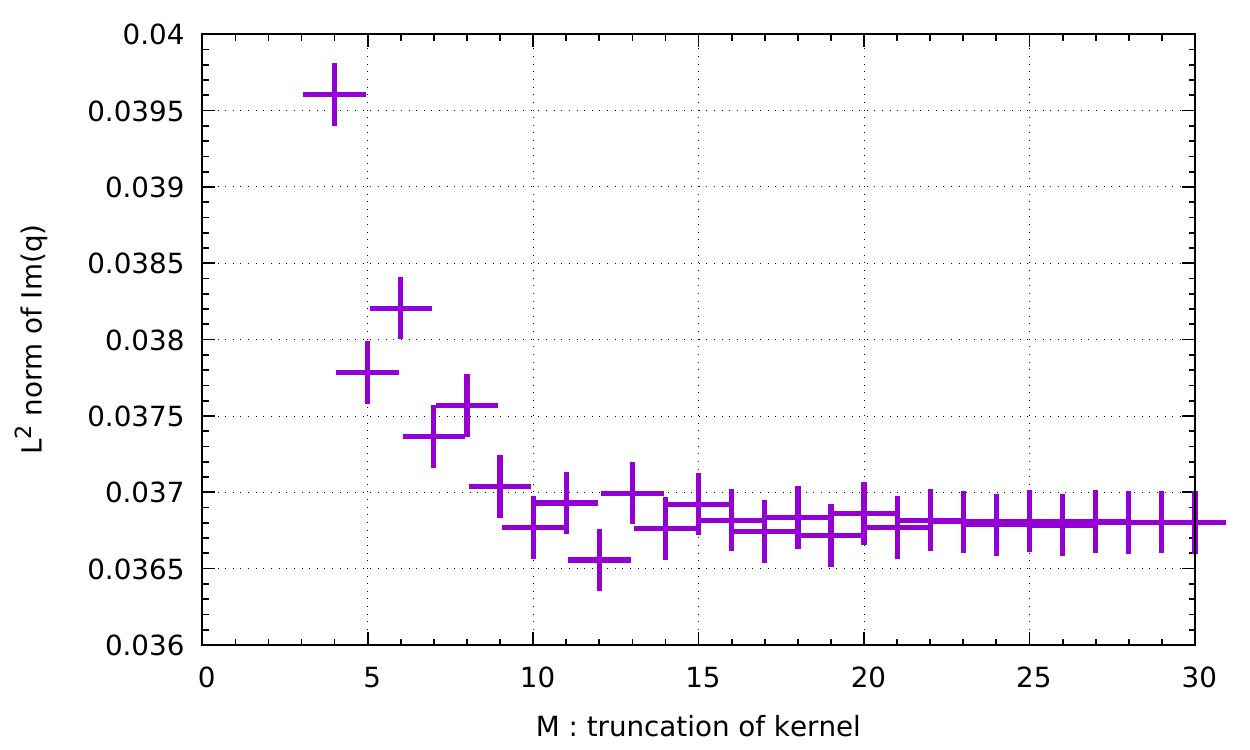}
\caption{\label{fig:M}$\norm{\Impart q}_2$ of numerically reconstructed source. (Left) CSIEs are solved by the cut-off method with $\epsilon=0$ yielding an optimal parameter $M=10$. 
(Right) CSIEs are solved by the ex-log method with $\epsilon=0.01$ yielding an optimal parameter $M=12$.}
\end{figure}

CSIEs in \cref{step:solveCSIE1} and \cref{step:ellipticCauchy} share the same coefficient matrix $(1+\epsilon)E_N - iH_N$ and thus it enables us to use LU decomposition for saving computational time. We take $N=2{,}000$ so that $\Delta x = 0.001$ is close to $\Delta\omega = \pi/3{,}141$. 
Figure \ref{fig:J-173} depicts numerical solutions $\mathcal{J}^{C}_{-173,n}$ by the cut-off method with $\epsilon=0$ and by the ex-log method with $\epsilon=0.01$ as the first iteration, and they show similar trends. The results using the ex-log method with $\epsilon=0$ are illustrated in Figure~\ref{fig:J-173e0}. 
Note that both the real and imaginary parts in Figure~\ref{fig:J-173e0} oscillate in seriously wider vertical ranges than those in Figure~\ref{fig:J-173}.
It shows that the ex-log method with $\epsilon=0$ is worse ill-conditioned and harder to solve accurately.
Numerical solutions $\mathcal{I}^{C}_{-1,n}$ and $\mathcal{I}^{C}_{0,n}$ as the final iterations in \cref{step:ellipticCauchy}
are shown in Figure~\ref{fig:I0}, and contrasted with those (in orange) computed directly using the numerical solution of the forward problem . Their real parts are similar characteristics, and relative magnitudes of their imaginary parts to real parts are also equivalently small.
Figure~\ref{fig:J-173} and Figure~\ref{fig:I0} also indicate that numerical solutions by the ex-log method have sharp peaks near the edges, while the cut-off method generates milder variation.
\begin{figure}[ht]
  \includegraphics[width=.48\textwidth]{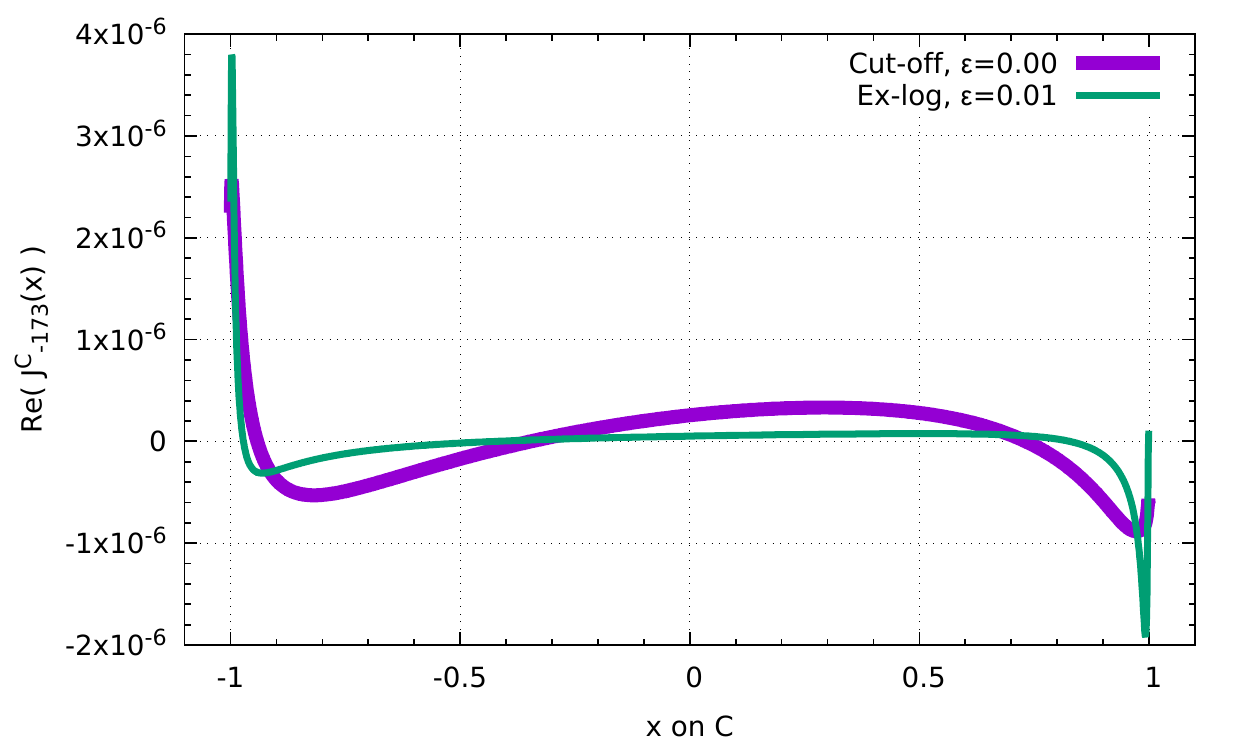}
  \includegraphics[width=.48\textwidth]{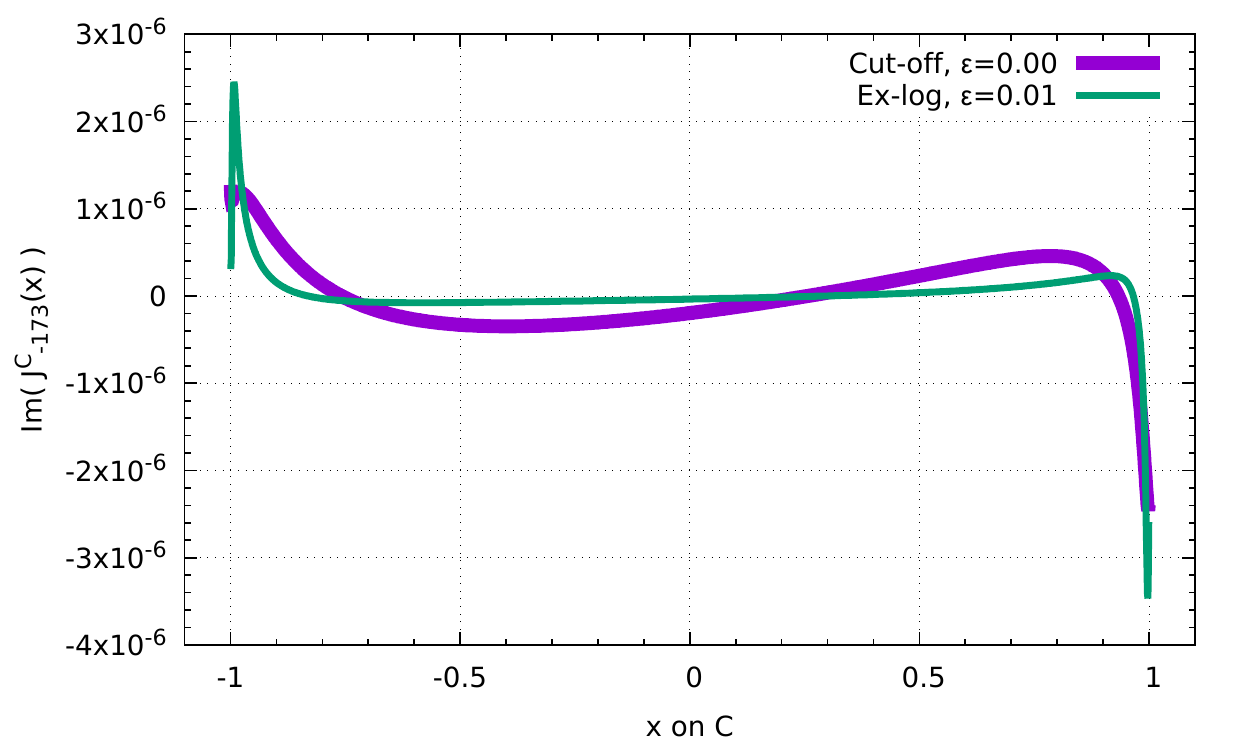}
\caption{\label{fig:J-173}Numerical solutions to CSIE, $\mathcal{J}^{C}_{-173,n}$ in \cref{step:solveCSIE1}. The left and right figures respectively show their real and imaginary parts.}
\end{figure}
\begin{figure}[ht]
  \includegraphics[width=.48\textwidth]{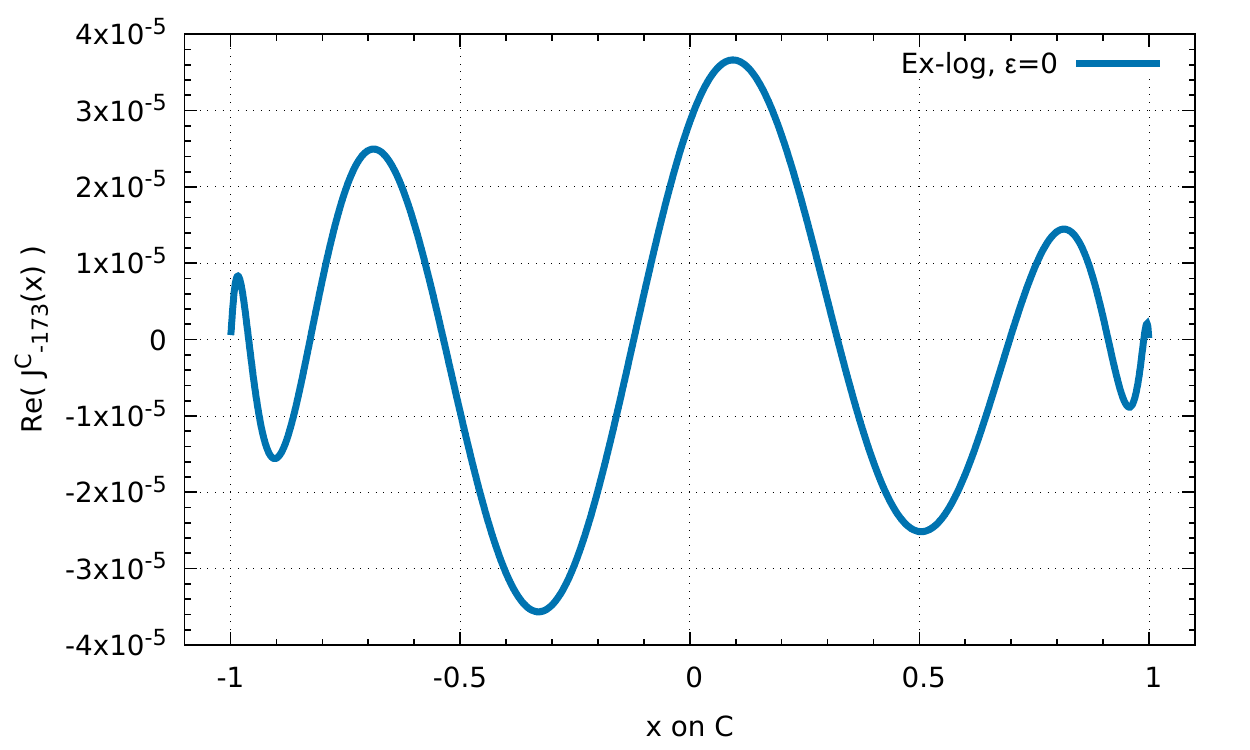}
  \includegraphics[width=.48\textwidth]{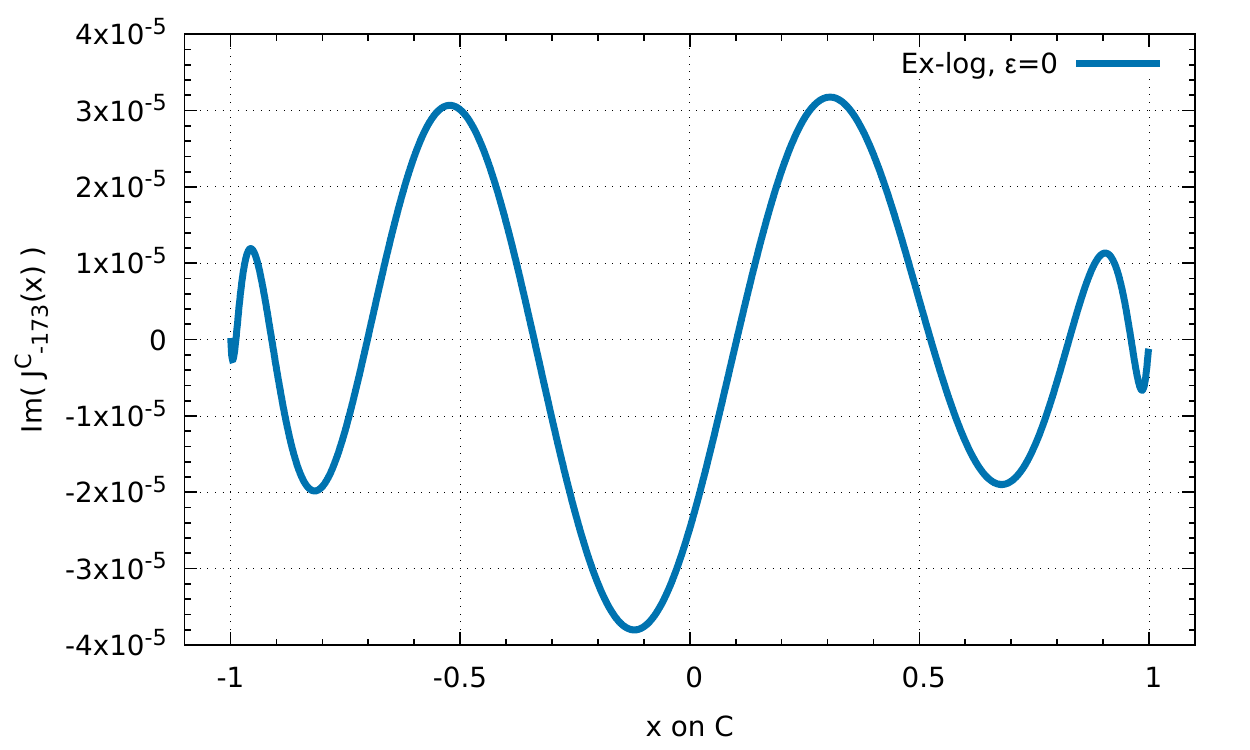}
\caption{\label{fig:J-173e0}The real (left) and imaginary (right) part of the numerical solutions to CSIE, $\mathcal{J}^{C}_{-173,n}$ in \cref{step:solveCSIE1} by the ex-log method with $\epsilon=0$.}
\end{figure}
\begin{figure}[ht]
  \includegraphics[width=.48\textwidth]{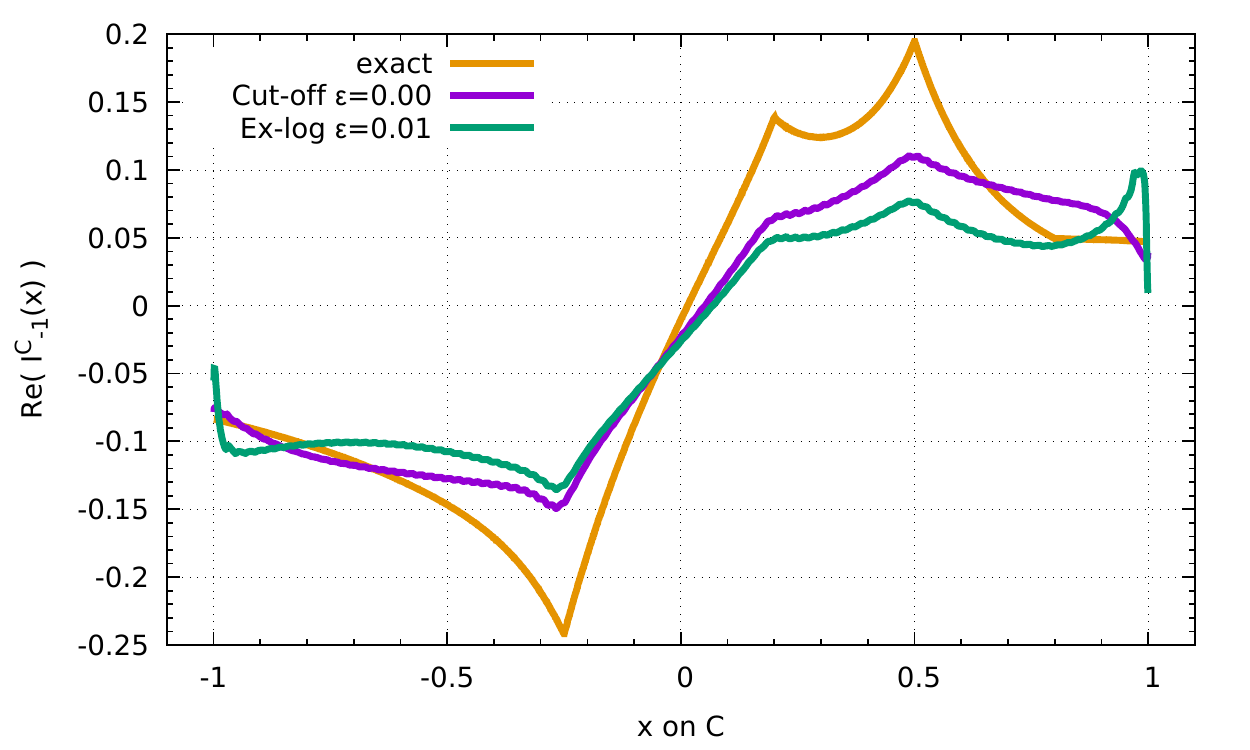}
  \includegraphics[width=.48\textwidth]{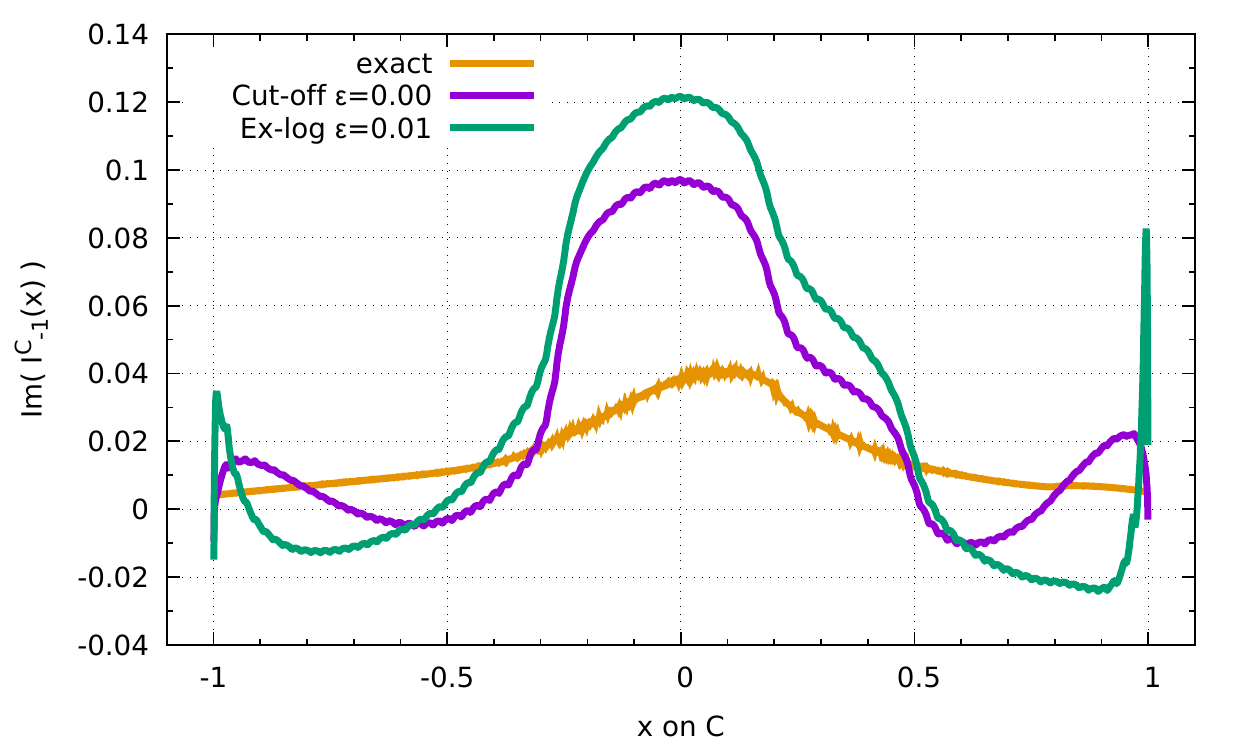}

  \includegraphics[width=.48\textwidth]{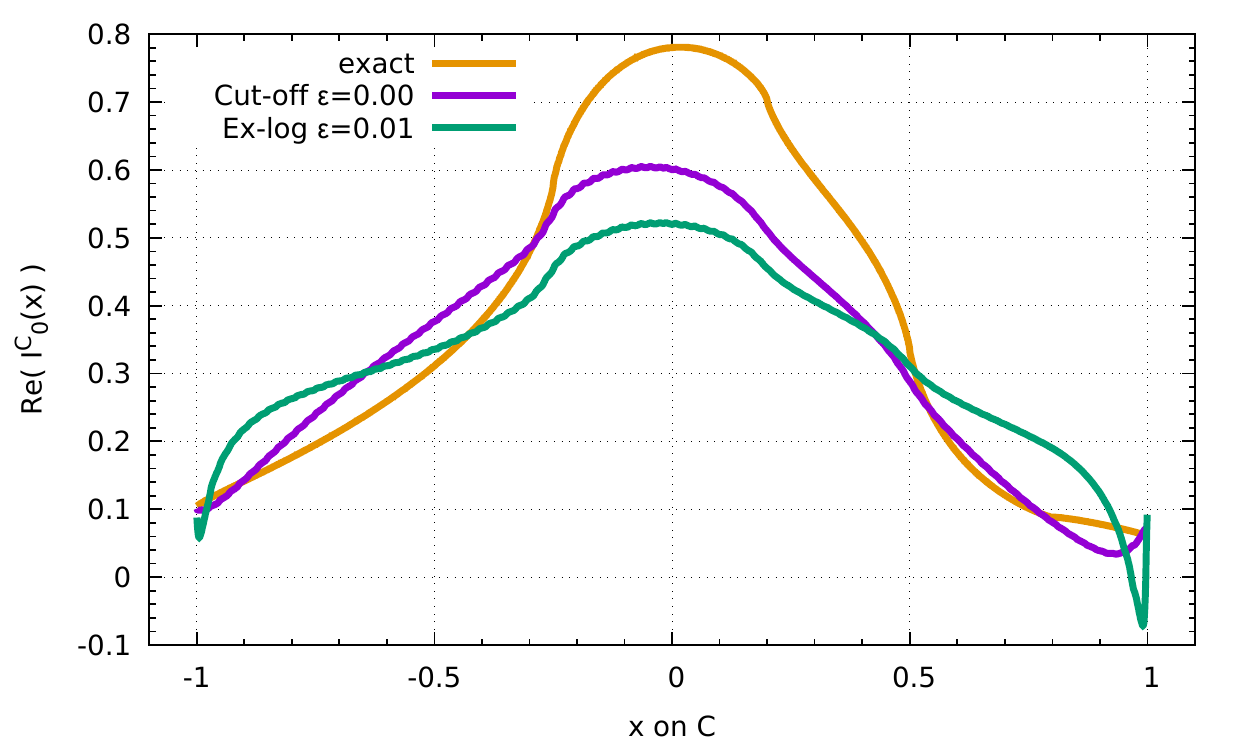}
  \includegraphics[width=.48\textwidth]{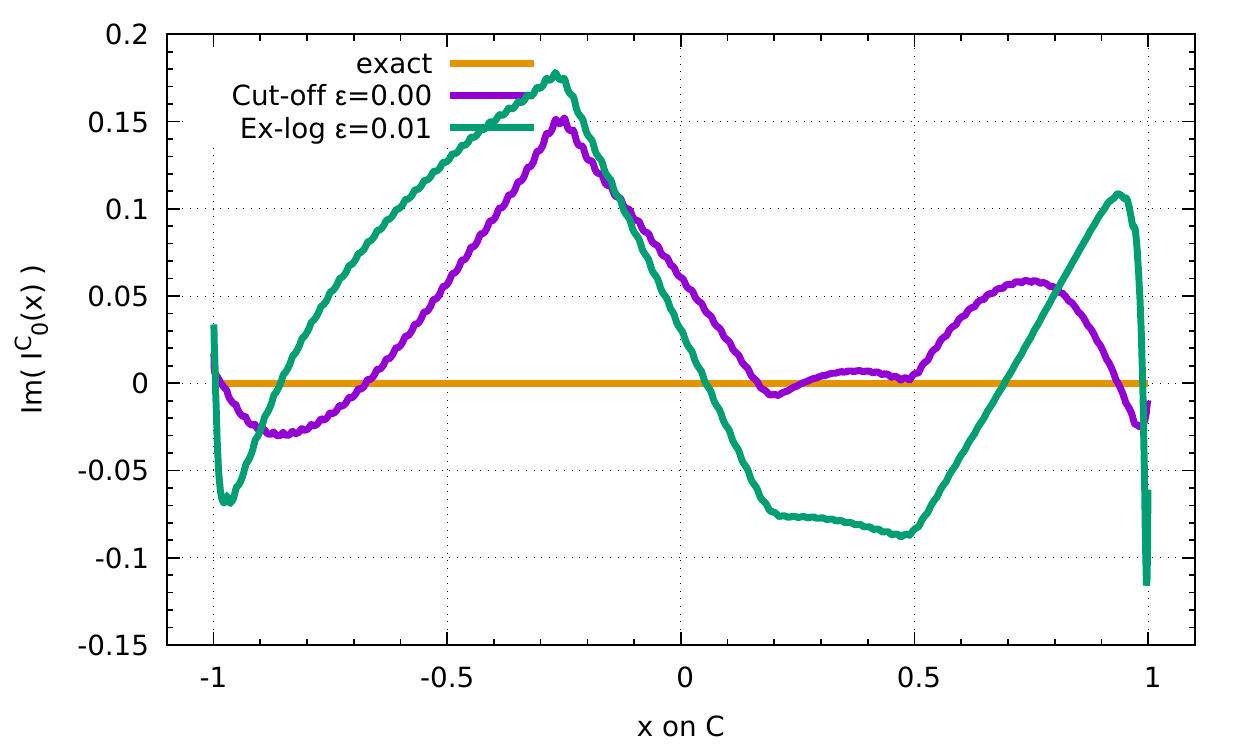}
\caption{\label{fig:I0}Numerical solutions to CSIE, $\mathcal{I}^{C}_{-1,n}$ (upper) and $\mathcal{I}^{C}_{0,n}$ (lower) in \cref{step:ellipticCauchy}. The left and right figures respectively show the real and imaginary parts.}
\end{figure}

\begin{figure}[ht]
\begin{minipage}{.7\textwidth}
  \includegraphics[width=\textwidth,bb=75 50 355 185]{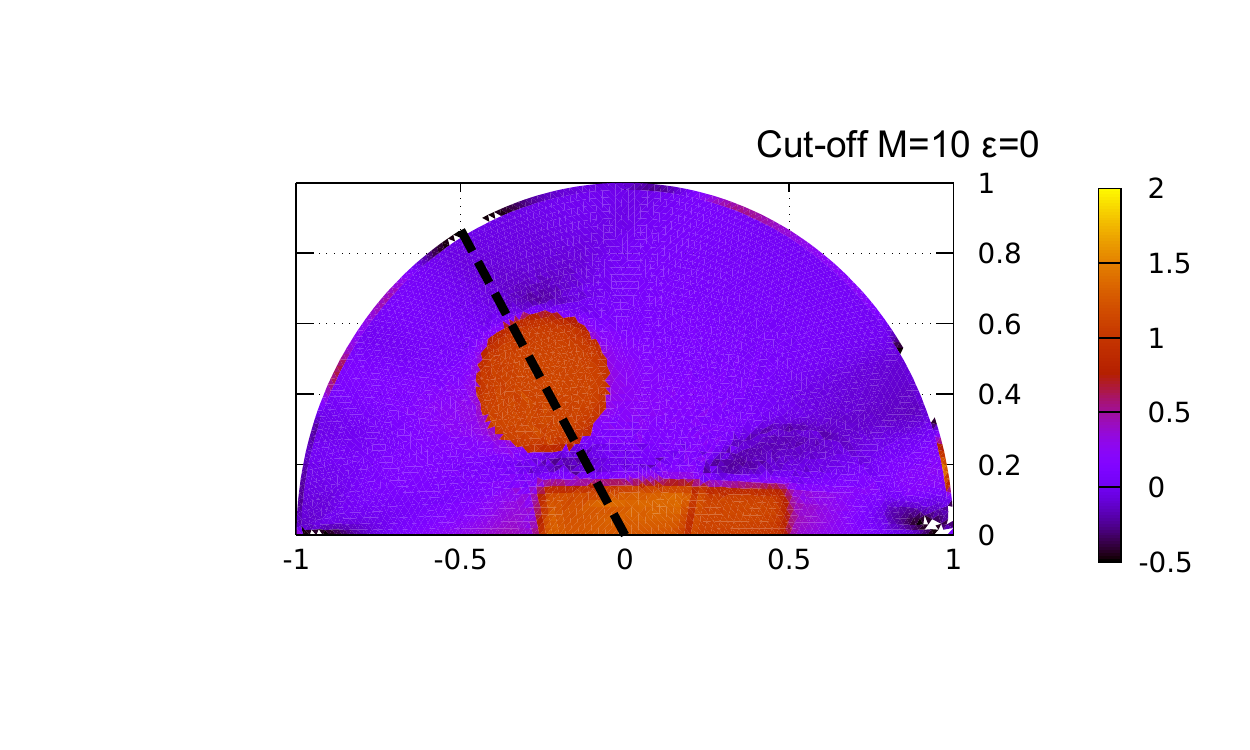}
\end{minipage}
\begin{minipage}{.25\textwidth}
  \includegraphics[width=\textwidth,bb=55 0 190 195]{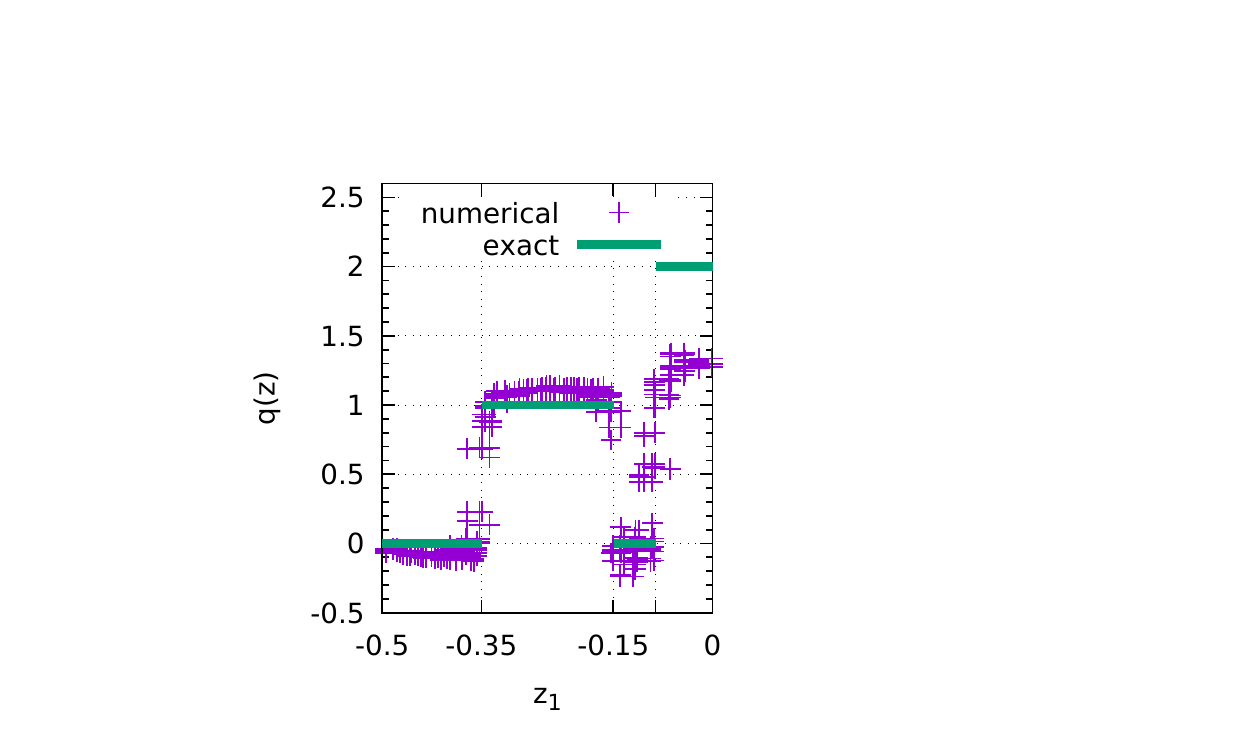}
\end{minipage}
\caption{\label{fig:qcutoff}Numerical reconstruction of the source $q$ by the cut-off method in \eqref{eq:inteq:d} without $\epsilon$-regularization. On the right is the reconstructed section on the dotted segment.}
\end{figure}

\begin{figure}[ht]
\begin{minipage}{.7\textwidth}
  \includegraphics[width=\textwidth,bb=75 50 355 185]{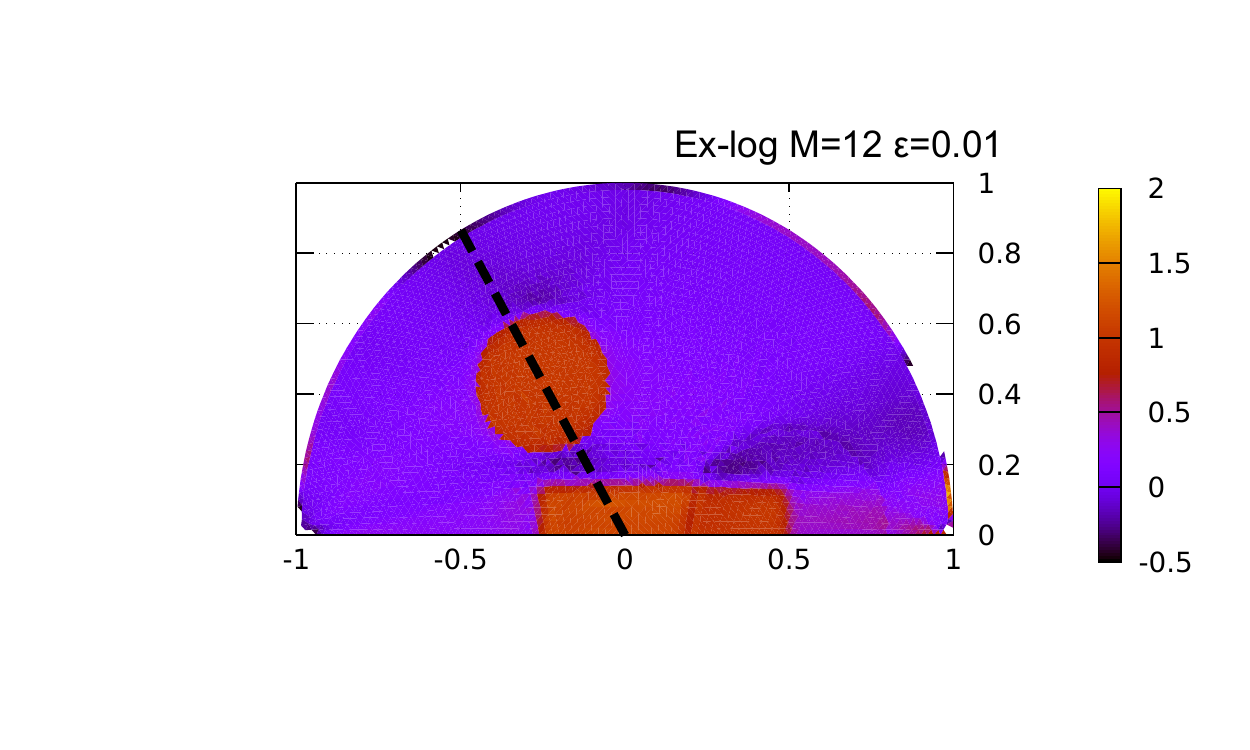}
\end{minipage}
\begin{minipage}{.25\textwidth}
  \includegraphics[width=\textwidth,bb=55 0 190 195]{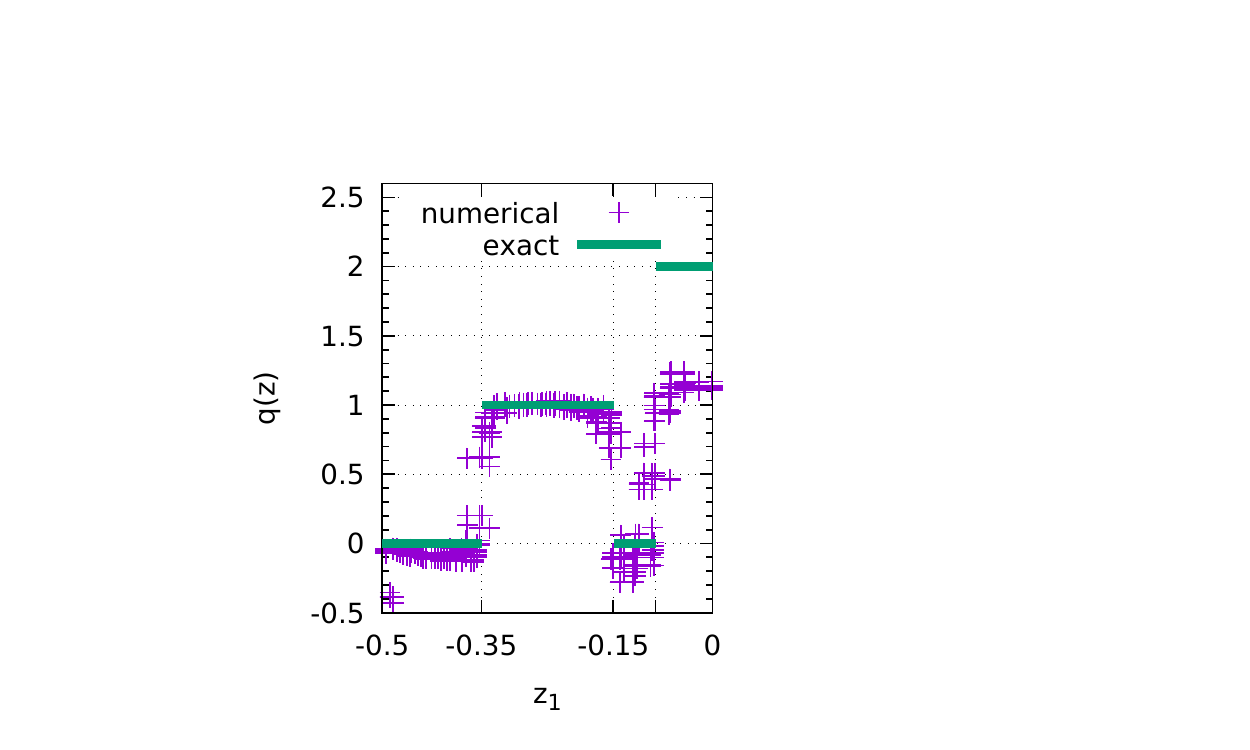}
\end{minipage}
\caption{\label{fig:qlog}Numerical reconstruction of the source $q$ by \eqref{eq:inteq:d2} with $\epsilon=0.01$. The right graph is its section on the dotted segment.}
\end{figure}

In the cut-off method, the computation of $\mathcal{P}^{-}_{-m}$ in \cref{step:solveCSIE1} took $233$ seconds, while that of $\mathcal{J}^{\Omega^{+}}_{-m}$ in \cref{step:BoundaryIntegral} took $1{,}093$ seconds. These computations consume the bulk of the total computing time (of $1{,}393$ seconds). However, due to the pointwise nature of the reconstruction, the proposed algorithm  is amenable to parallelization. By using $48$ threads OpenMP parallel computation, the total times for computing $\mathcal{P}^{-}_{-m}$ in \cref{step:solveCSIE1} and $\mathcal{J}^{\Omega^{+}}_{-m}$ in \cref{step:BoundaryIntegral} are reduced to $11$ and $39$ seconds respectively, and the total time to $62$ seconds.
The numerically reconstructed source $q$ in $\Omega^+$ by the cut-off method with $\epsilon=0$ is shown in Figure~\ref{fig:qcutoff}, and by the ex-log method with $\epsilon=0.01$
is shown in Figure~\ref{fig:qlog}. The corresponding sections on the dotted line are shown in the same figure on the right. We note that the singular support of the internal sources is clearly identified, while the reconstructed values are quantitatively reasonable.


\section*{Acknowledgment}
Authors wish to thank Professor Keith Matthews for fruitful comments on the backgrounds of Montgomery-Matthews' inequality.
The work of H.~Fujiwara was supported by JSPS KAKENHI Grant Numbers JP20H01821 and 21H00999.
The work of K.~ Sadiq  was supported by the Austrian Science Fund (FWF), Project P31053--N32, and by the FWF Project F6801--N36 within the Special Research Program SFB F68 ``Tomography Across the Scales''. 
The work of A.~Tamasan  was supported in part by the NSF grant DMS-1907097.
 
\bibliographystyle{siamplain}

\end{document}